\documentclass[11pt,english]{smfart}

\usepackage[dvipdfm]{graphicx}
\usepackage{epsfig}
\usepackage[cmtip,arrow]{xy}
\usepackage{pb-diagram,pb-xy}
\usepackage[dvipdfm]{graphicx}
\usepackage{amsmath}
\usepackage{amssymb}
\usepackage{amscd}
\usepackage{latexsym}
\usepackage{tabularx}
\usepackage{a4wide}
\usepackage{color}
\usepackage{subfigure}

\small\normalsize

\newtheorem{theo}{Theorem}[section]
\newtheorem{coro}[theo]{Corollary}
\newtheorem{prop}[theo]{Proposition}
\newtheorem{lemm}[theo]{Lemma}

\newtheorem{claim}{Claim}
\newtheorem{defi}[theo]{Definition}
\newtheorem{rema}[theo]{Remark}

\newcommand{\BibTeX}{{\scshape Bib}\kern-.08em\TeX}
\newcommand{\T}{\S\kern .15em\relax }
\newcommand{\AMS}{$\mathcal{A}$\kern-.1667em\lower.5ex\hbox
        {$\mathcal{M}$}\kern-.125em$\mathcal{S}$}

\newtheorem{observ}{Observation}
\newtheorem{question}{Question}

\newcommand{\Vor}{\mathrm{Vor}}
\newcommand{\Crit}{\mathrm{Crit}}
\newcommand{\Ext}{\mathrm{Ext}}
\newcommand{\Pic}{\mathrm{Pic}}
\newcommand{\Gr}{\mathrm{Gr}}
\newcommand{\Del}{\mathrm{Del}}
\newcommand{\conv}{\mathrm{Conv}}
\newcommand{\Cov}{\mathrm{Cov}}
\newcommand{\A}{\mathcal A}
\newcommand{\B}{\mathcal B}

\title[Riemann-Roch for Sub-Lattices of the Root Lattice $A_n$]{Riemann-Roch for Sub-Lattices \\ of the Root Lattice $A_n$}

\author{Omid Amini}
\address{CNRS - DMA, \'Ecole Normale Sup\'erieure, Paris, France}
\author{Madhusudan Manjunath}
\address{Max-Planck Institut f\"ur Informatik, Saarbr\"ucken, Germany}
\email{oamini@math.ens.fr, manjun@mpi-inf.mpg.de}

\begin{document}
\maketitle

\begin{abstract}
Recently, Baker and Norine ({\it Advances in Mathematics}, 215({\bf 2}): 766-788, 2007) found new analogies between graphs and Riemann surfaces by developing a Riemann-Roch machinery on a finite graph $G$. In this paper, we develop a general Riemann-Roch Theory for sub-lattices of the root lattice $A_n$ by following the work of Baker and Norine, and establish connections between the Riemann-Roch theory and the Voronoi diagrams of lattices under certain simplicial distance functions. In this way, we rediscover the work of Baker and Norine from a geometric point of view and generalise their results to other sub-lattices of $A_n$. In particular, we provide a geometric approach for the study of the Laplacian of graphs. 
We also discuss some problems on classification of lattices with a Riemann-Roch formula as well as some related algorithmic issues. 
\end{abstract}
\tableofcontents
\section{Introduction}
\label{sec:intro}

 Recently, Baker and Norine~\cite{BaNo07} proved a graph theoretic analogue of the classical Riemann-Roch theorem for curves in algebraic geometry. The proof is combinatorial and makes use of chip-firing games~\cite{BLS91} and parking functions on graphs. Several papers later extended the results of Baker and Norine to tropical curves~\cite{GK08,HKN07,MZ07}. The question treated in this paper is to characterize those lattices which admit a Riemann-Roch theorem for the corresponding analogue of the rank-function defined by Baker and Norine.   

\subsection{Chip-Firing Game} Let $G=(V,E)$ be a finite connected (multi-)graph with the set of vertices $V$ and  the set of edges $E$. We suppose that $G$ does not have loops. The chip-firing game is the following game played on the set of vertices of $G$: At the initial configuration of the game, each
vertex of the graph is assigned an integer number of chips. A vertex can have a positive number of chips in its possession or can be assigned a negative number meaning that the vertex is in debt with the amount described by the absolute value of that number. At each step of the chip-firing game, a vertex in the graph can decide to \emph{fire}: firing means the vertex gives one chip along each edge incident with it, to its neighbours. Thus, after the firing made by a vertex $v$ of degree $d_v$, the integer assigned to $v$ decreases by $d_v$, while the integer associated to each vertex $u$ connected by $k_u$ (parallel) edges to $v$ increases by $k_u$. The objective of the vertices of the graph is to come up with a configuration in which no vertex is in debt, i.e., a configuration in which all the integers associated to vertices become non-negative. 

\vspace{.2cm}

\noindent \textbf{Problem.} \emph{Given an intial configuration, is there a finite sequence of chip-firings such that
eventually each vertex has a non-negative number of chips?}
 
 \vspace{.2cm}

 Let $\deg(\mathcal C)$, degree of $\mathcal C$,  be the total number of chips present in the game, i.e., the sum of the integers associated to the vertices of the graph. It is clear that degree remains unchanged through each step of the game, thus, a necessary condition for a positive answer to the above question is to have a non-negative degree.
 
\subsection{Riemann-Roch Theorem For Graphs} To each given chip-firing configuration $\mathcal C$, Baker and Norine associate a rank $r(\mathcal C)$ as follows. The rank of $\mathcal C$ is $-1$ if there is no way to obtain a configuration in which all the vertices have non-negative weights. And otherwise, $r(\mathcal C)$ is the maximum non-negative integer $r$ such that removing any set of $r$ chips from the game (in an arbitrary way), the obtained configuration can be still transformed via a sequence of chip-firings to a configuration where no vertex is in debt. In particular, note that $r(\mathcal C) \geq 0$ iff there is a sequence of chip-firings which results in a configuration with non-negative number of chips at each vertex.

 The main theorem of~\cite{BaNo07} is a duality theorem for the rank function $r(.)$. Let $\mathcal K$ be the \emph{canonical} configuration defined as follows: $\mathcal K$ is the configuration of chips in which every vertex $v$ of degree $d_v$ is assigned $d_v -2$ chips. Given a chip-firing configuration $\mathcal C$, the configuration $\mathcal K \setminus \mathcal C$ is defined as follows: a vertex $v$ of degree $d_v$ is assigned $d_v -2 - c_v$ chips in $\mathcal K \setminus \mathcal C$ if $v$ is assigned $c_v$ chips in $\mathcal C$. 

\noindent Recall that the genus $g$ of a connected graph $G$ with $n+1$ vertices and $m$ edges is $g := m-n$. 
\begin{theo}[Riemann-Roch theorem for graphs; Baker-Norine~\cite{BaNo07}]\label{BaNo07}
For every configuration $\mathcal C$, we have
\[r(\mathcal C) - r(\mathcal K \setminus \mathcal C) = \deg(\mathcal C) - g+1\, .\] 
\end{theo}

\vspace{.2cm}

The existing proof of the Riemann-Roch theorem for graphs (and its extension to metric graphs and tropical curves~\cite{GK08,HKN07,MZ07}) is based on a family of specific configurations which are called reduced. We refer to~\cite{BaNo07} for more details, explaining the origin of the name given to this theorem in its connections with the Riemann-roch theorem for algebraic curves. 

\noindent Here we just cite some direct consequences of the above theorem for the chip-firing game. 
\begin{itemize}
\item If a configuration $\mathcal C$ contains more than $g$ chips, there is a sequence of chip-firings which produces a configuration where no vertex is in debt (more generally, one has $r(\mathcal C) \geq \deg(\mathcal C) -g$).
\item $r(\mathcal K) = g-1$ (note that $\deg(K) = 2g-2$).
\end{itemize}

\subsection{Reformulation in Terms of the Laplacian Lattice}Recall that a lattice is a discrete subgroup of the abelian group $(\mathbb R^{n},+)$ for some integer $n$ (e.g., the lattice $\mathbb Z^n \subset \mathbb R^n$), and the rank of a lattice is its rank considered as a free abelian group. A sub-lattice of $\mathbb Z^n$ is called integral in this paper.

Let $G=(V,E)$ be a given undirected connected (multi-)graph and $V=\{v_0,\dots,v_n\}$. The {\it Laplacian} of $G$ is the matrix $Q=D-A$, where $D$ is the diagonal matrix whose $(i,i)-$th entry is the degree of $v_i$, and $A$ is the {\it adjacency matrix} of $G$ whose $(i,j)-$th entry is the number of edges between $v_i$ and $v_j$. It is well-known and easy to verify that $Q$ is symmetric, has rank $n$, and that the kernel of $Q$ is spanned by the vector whose entries are all equal to $1$, c.f.~\cite{Bi93}.

\noindent The {\it Laplacian lattice} $L_G$ of $G$ is defined as the image of $\mathbb Z^{n+1}$ under the linear map defined by $Q$, i.e., $L_G:=Q(\mathbb Z^{n+1})$,  c.f., \cite{BHN97}. Since $G$ is a connected graph, $L_G$ is a sub-lattice of the root lattice $A_n$ of full-rank equal to $n$, where $A_n \subset \mathbb R^{n+1}$ is the lattice defined as follows\footnote{Root refers here to root systems in the classification theory of simple Lie algebras~\cite{Bourbaki}}:
$$A_n := \Bigl\{\,x = (x_0,\dots,x_n)\in \mathbb Z^{n+1} \ | \ \sum x_i =0\,\Bigr\}.$$
 Note that $A_n$ is a discrete sub-group of the hyperplane 
 $$H_0= \Bigl\{\,x = (x_0,\dots,x_n)\in \mathbb R^{n+1}|\sum x_i =0\,\Bigr\}$$ 
 of $\mathbb R^{n+1}$ and has rank $n$ .

To each configuration $\mathcal C$, it is straightforward to associate a point $D_\mathcal C$ in $\mathbb Z^{n+1}$: $D_\mathcal C$ is  the vector with coordinates equal to the number of chips given to the vertices of $G$.  For a sequence of chip-firings on $\mathcal C$ resulting in another configuration $\mathcal C'$, it is easy to see that there exists a vector $v \in L_G$ such that $D_{\mathcal C'} = D_{\mathcal C} +v$. Conversely,  if $D_{\mathcal C'} = D_{\mathcal C}+v$ for a vector $v \in L_G$, then there is a sequence of chip-firings transforming $\mathcal C$ to $\mathcal C'$. Using this equivalence, it is possible to transform the chip-firing game and the statement of the Riemann-Roch theorem to a statement about $\mathbb Z^{n+1}$ and the Laplacian lattice $L_G \subset A_n$.
\begin{rema}\rm
Laplacian of graphs and their spectral theory have been well studied. The Laplacian captures information about the geometry and combinatorics of the graph $G$, for example, it provides bounds on the expansion of $G$ (we refer to the survey~\cite{HLW06}) or on the quasi-randomness properties of the graph, see~\cite{Ch97}. 
The famous Matrix Tree Theorem states that {\it the cardinality of the (finite) Picard group $\mathrm{Pic}(G) := A_n/L_G $ is the number of spanning trees of $G$}. 
\end{rema}
\subsection{Linear Systems of Integral Points and the Rank Function}
Let $L$ be a sub-lattice of $A_n$ of full-rank (e.g., $L=L_G$).
Define an equivalence relation $\sim$ on the set of points of $\mathbb{Z}^{n+1}$ as follows:    
$D \sim D'$ if and only if $D-D' \in L$.  This equivalence relation is referred to as {\it linear equivalence} and the equivalence classes are denoted by $\mathbb{Z}^{n+1}/L_G$. 
We say that a point $E$ in $\mathbb Z^{n+1}$ is {\it effective} or {\it non-negative}, if all the coordinates are non-negative. For a point $D \in \mathbb Z^{n+1}$, the linear system associated to $D$ is the set $|D|$ of all effective points linearly equivalent to $D$: $$|D| = \Bigl\{\:E \in \mathbb Z^{n+1}\::\: E \geq 0,\: \: E \sim D\:\Bigr\}.$$

The {\it rank} of an integral point $D\in\mathbb Z^{n+1}$, denoted by $r(D)$, is defined by setting $r(D)=-1$, if $|D|=\emptyset$, and then declaring that for each integer $s\geq0$, $r(D)\geq s$ if and only if $|D-E|\neq \emptyset$ for all effective integral points $E$ of degree $s$. Observe that $r(D)$ is well-defined and only depends on the linear equivalence class of $D$. Note that $r(D)$ can be defined as follows:
$$r(D)=\min\:\Bigl\{\:\deg(E)\:|\:|D-E|=\emptyset,\: E \geq 0\Bigr\}-1.$$
Obviously, $\deg(D)$ is a trivial upper bound for $r(D)$.

\subsection{Extension of the Riemann-Roch Theorem to Sub-lattices of $A_n$}
The main aim of this paper is to provide a characterization of the sub-lattices of $A_n$ which admit a Riemann-Roch theorem with respect to the rank-function defined above. In the mean-while, our approach provides a geometric proof of the theorem of Baker and Norine, Theorem~\ref{BaNo07}. 

We show that Riemann-Roch theory associated to a full rank sub-lattice $L$ of $A_n$ is related to the study of the Voronoi diagram of the lattice $L$ in the hyperplane $H_0$ under a certain simplicial distance function. The whole theory is then captured by the corresponding critical points of this simplicial distance function. 

\noindent We associate two geometric invariants to each such sub-lattice of $A_n$, the {\it min-} and the {\it max-genus}, denoted respectively by $g_{min}$ and $g_{max}$. 
Two main characteristic properties for a given sub-lattice of $A_n$ are then defined. The first one is what we call {\it Reflection Invariance}, and one of our results here is a weak Riemann-Roch theorem for reflection-invariant sub-lattices of $A_n$ of full-rank $n$.
\begin{theo}[Weak Riemann-Roch] \label{rrineq_theo} 
Let $L$ be a reflection invariant sub-lattice of $A_n$ of rank $n$. There exists a point $K\in\mathbb Z^{n+1}$, called \emph{canonical}, such that for every point $D \in \mathbb{Z}^{n+1}$, we have
\begin{equation*}3g_{min}-2g_{max}-1\: \leq\: r(K-D)-r(D)+\deg(D)\: \leq\:g_{max}-1\:	.
\end{equation*}
\end{theo}

\vspace{.2cm}

The second characteristic property is called {\it Uniformity} and simply means $g_{min} = g_{max}$. It is straightforward to derive a Riemann-Roch theorem for uniform reflection-invariant sub-lattices of $A_n$ of rank $n$ from Theorem~\ref{rrineq_theo} above.  
\begin{theo}[Riemann-Roch] \label{rr_theo} 
Let $L$ be a uniform reflection invariant sub-lattice of $A_n$. Then there exists a point $K\in\mathbb Z^{n+1}$, called \emph{canonical}, such that for every point $D \in \mathbb{Z}^{n+1}$, we have
\begin{equation*} r(K-D)-r(D)+\deg(D) = g-1,
\end{equation*}
where $g = g_{min}=g_{max}$.
\end{theo}

 We then show that Laplacian lattices of undirected connected graphs are uniform and reflection invariant, obtaining a geometric proof of the Riemann-Roch theorem for graphs. As a consequence of our results, we provide an explicit description of the Voronoi diagram of lattices generated by Laplacian of connected graphs and discuss some duality concerning the  arrangement of simplices defined by the points of the Laplacian lattice. 
 
In the case of the Laplacian lattices of connected regular digraphs,  we also provide a slightly  stronger statement than Theorem~\ref{rrineq_theo} above. 
 
   The above results also provide a characterization of full-rank sub-lattices of $A_n$ for which a Riemann-Roch formula holds, indeed, these are exactly those lattices which have the uniformity and the reflection-invariance properties. We conjecture that any such lattice is the Laplacian lattice of an oriented multi-graph (as we will see, there are examples of such  lattices which are not the Laplacian lattice of  any unoriented multi-graph).
\subsection{Organisation of the paper.} The paper is structured as follows. Sections~\ref{sec:prel} and~\ref{app:theo-dom} provide the preliminaries. This includes the definition of a geometric region in $\mathbb R^{n+1}$ associated to a given lattice, called the Sigma-region, some results on the shape of this region in terms of the extremal points, and the definition of the min- and max-genus. In Section~\ref{sec:voronoi}, we provide the geometric terminology we need in the following sections for the proof of our main results. This is done in terms of a certain kind of Voronoi diagram, and in particular, some main properties of the Voronoi diagram of sub-lattices of $A_n$ under a certain simplicial distance function are provided in this section. The proof of our Riemann-Roch theorem is provided in Section~\ref{sec:R-R-theo}.  Most of the geometric terminology introduced in the first sections will be needed to define an involution on the set of extremal points of the Sigma-Region, the proof of the Riemann-Roch theorem is then a direct consequence of this and the definition of the min- and max-genus.  It is helpful to note that the main ingredients used directly in the proof of Theorems~\ref{rrineq_theo} and~\ref{rr_theo} are the results of Section~\ref{sec:prel} and Lemma~\ref{crit_lem} (and its Corollary~\ref{cor:crit}). The results of the first sections are then used in treating the examples in Section~\ref{sec:examples}, specially for the Laplacian lattices. We derive in this section a new proof of the main theorem of~\cite{BaNo07}, the Riemann-Roch theorem for graphs.

\noindent Our work raises questions on the classification of sub-lattices of $A_n$ with reflection invariance and/or uniformity properties. In Section~\ref{sec:examples}, we present a complete answer for sub-lattices of $A_2$.
 \noindent Finally, some algorithmic questions are discussed in Section~\ref{sec:algo}, e.g., we show that it is computationally hard to decide if the rank function is non-negative at a given point for a general sub-lattice of $A_n$. This is interesting since in the case of Laplacian lattices of graphs, the problem of deciding if the rank function is non-negative can be solved in polynomial time.

 As we said, in what follows we will assume that $L$ is an integral sub-lattice in $H_0$ of full-rank, i.e., a sub-lattice of $A_n$. But indeed, what we are going to present also works in the more general setting of full rank sub-lattices of $H_0$, though the invariants and rank function defined for these lattices are not integer. We will say a few words on this and some other results in the concluding section.
   
\subsection*{Basic Notations}
A point of $\mathbb R^{n+1}$ with integer coordinates is called an {\it integral point}. By a lattice $L$, we mean a discrete subgroup of $H_0$ of maximum rank. Recall that $H_0$ is the set of all points of $\mathbb R^{n+1}$ such that the sum of their coordinates is zero. The elements of $L$ are called {\it lattice points}. The positive cone in $\mathbb R^{n+1}$ consists of all the points with non-negative coordinates. We can define a partial order in $\mathbb R^{n+1}$ as follows: $a\leq b$ if and only if $b-a$ is in the positive cone, i.e., if each coordinate of $b-a$ is non-negative. In this case we say $b$ {\it dominates} $a$. Also we write $a<b$ if all the coordinates of $b-a$ are strictly positive.   

For a point $v=(v_0,\dots,v_n) \in \mathbb R^{n+1}$, we denote by $v^-$ and $v^+$ the negative and positive parts of $v$ respectively. 
For a point $p=(p_0,\dots,p_n) \in \mathbb{R}^{n+1}$, we define the degree of $p$ as $\deg(p)=\sum_{i=0}^{n}p_i$. 
For each $k$, by $H_k$ we denote the 
 hyperplane consisting of points of degree $k$, i.e.,
$H_k = \{x \in \mathbb R^{n+1}\ | \ \deg(x)=k\}.$
 By $\pi_k$, we denote the projection from $\mathbb R^{n+1}$ onto $H_k$ along $\vec{\mathbf{1}}=(1,\dots,1)$. In particular, $\pi_0$ is the projection onto $H_0$.	
Finally for an integral point $D\in\mathbb Z^{n+1}$, by $N(D)$ we denote the set of all neighbours of $D$ in $\mathbb Z^{n+1}$, which consists of all the points of $\mathbb Z^{n+1}$ which have distance at most one to $D$ in $\ell_{\infty}$ norm.

 In the following, to simplify the presentation, we will use the convention of tropical arithmetic, briefly recalled below. (This is so only  a matter of notation). The tropical semiring $(\mathbb R,\oplus,\otimes)$ is defined as follows: As a
set this is just the real numbers $\mathbb R$. However, one redefines the basic arithmetic operations of
addition and multiplication of real numbers as follows:
$x\:\oplus\:y := \min\:(x,y) \:\ \  \textrm{and}\ \ \  x\:\otimes\:y := x + y.$
In words, the tropical sum of two numbers is their minimum, and the tropical
product of two numbers is their sum. We can extend the tropical sum and the tropical product to vectors by doing the operations coordinate-wise. 

\section{Preliminaries}
\label{sec:prel}
All through this section $L$ will denote a full rank (integral) sub-lattice of $H_0$.

\subsection{Sigma-Region of a Given Sub-lattice $L$ of $A_n$} Every point $D$ in $\mathbb Z^{n+1}$ defines two ``orthogonal'' cones in $\mathbb R^{n+1}$, denoted by $H^-_D$ and $H_D^+$, as follows:
 $H^{-}_D$ is the set of all points in $\mathbb R^{n+1}$ which are dominated by $D$. In other words 
\begin{center} $H^{-}_D=\{\:D' \ |\  D' \in \mathbb{R}^{n+1},  D-D'\geq 0\:\}.$ \end{center}
 Similarly $H^{+}_D$ is the set of points in $\mathbb R^{n+1}$ that dominate $D$. 
 In other words,
\begin{center}$H^{+}_D=\{\:D'\ | \ D' \in \mathbb{R}^{n+1}, D' - D\geq 0\:\}.$\end{center}
 For a cone $\mathcal C$ in $\mathbb R^{n+1}$, we denote by $\mathcal C(\mathbb Z)$ and $\mathcal C(\mathbb Q)$, the set of integral and rational points of the cone respectively. When there is no risk of confusion, we sometimes drop $(\mathbb Z)$ (resp. $(\mathbb Q)$) and only refer to $\mathcal C$ as the set of integral points (resp. rational points) of the cone $\mathcal C$.
\noindent The {\it Sigma-Region} of the lattice $L$ is, roughly speaking, the set of integral points of $\mathbb Z^{n+1}$ that are  
not contained in the cone $H^-_p$ for any point $p \in L$.  
More precisely:
\begin{defi}\rm 
The Sigma-Region of $L$, denoted by $\Sigma(L)$, is defined as follows:
  \begin{align*} 
  \Sigma(L)&=\{\: D \ | \ D\in\mathbb Z^{n+1}\ \&  \ \forall\: p \in L,\: D \nleq p\:\}\\
  & =  \mathbb Z^{n+1} \setminus \bigcup_{p \in L} H_p^-.
   \end{align*}
\end{defi} 

The following lemma shows the relation between the Sigma-Region and the rank of an integral point as defined in the previous section. 

\begin{lemm} \label{ranksig_lem} 
\noindent
\begin{itemize}
\item[(i)] For a point $D$ in $\mathbb{Z}^{n+1}$, $r(D)=-1$ if and only if $-D$ is a point in $\Sigma(L)$. 
\item[(ii)] More generally, $r(D)+1$ is the distance of $-D$ to $\Sigma(L)$ in the $\ell_1$ norm, i.e.,
\begin{align*}
 r(D) \:&\: = dist_{\ell_1}(-D,\Sigma(L))-1:=\inf\{||p+D||_{\ell_1}\ | \ p\in\Sigma(L)\}-1, 
\end{align*} 
 \textrm{where $||x||_{\ell_1} = \sum_{i=0}^n |x_i|$ for every point $x=(x_0,x_1,\dots,x_n)\in\mathbb R^{n+1}$}.
\end{itemize}
\end{lemm}

\begin{figure}[!htb]
\begin{center}
\vspace{-1cm}
\includegraphics[width=0.5\linewidth]{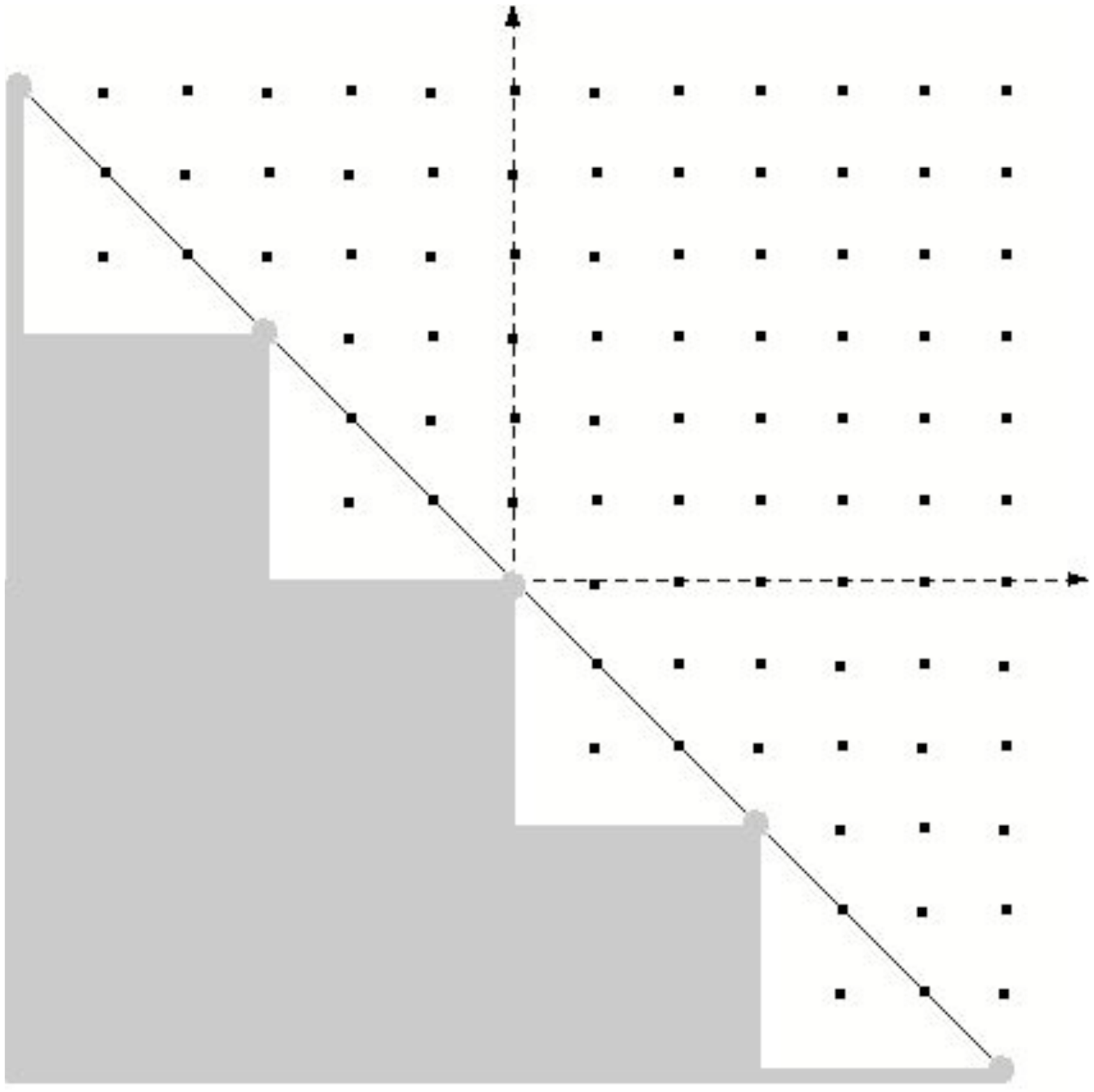}
\vspace{-1cm}
\caption{A finite portion of the Sigma-Region of a sub-lattice of $A_1$. All the black points  belong to the Sigma-Region. The integral points in the grey part are out of the Sigma-Region.}
\end{center}
\end{figure}

 Before presenting the proof of Lemma~\ref{ranksig_lem}, we need the following simple observation.
 \vspace{.5cm}
\begin{observ}~\label{observ-r-trick}
$\forall\: D_1,D_2 \in \mathbb Z^{n+1}$, we have $D_1 \in \Sigma(L) -D_2$ if and only if $D_2 \in \Sigma(L)-D_1$.
\end{observ}  

\noindent We shall usually use this observation without sometimes mentioning it explicitly.

\begin{proof}[Proof of Lemma~\ref{ranksig_lem}]
\noindent
\begin{itemize}
\item[(i)] Recall that $r(D)=-1$ means that  $|D|=\emptyset$. This in turn means that $D \ngeq p$ for any $p$ in $L$, or equivalently $-D \nleq q$ for any point $q$ in $L$ (because $L=-L$). We infer that $-D$ is a point of $\Sigma(L)$. Conversely, if $-D$ belongs to $\Sigma(L)$, then $-D \nleq q$ for any point $q$ in $L$, or equivalently $D \ngeq p$ for any $p$ in $L$ (because $L=-L$). This implies that $|D|=\emptyset$ and hence $r(D)=-1$.

\item[(ii)]Let $p^*$ be a point in  $\Sigma(L)$ which has minimum $\ell_1$ distance from $-D$, and define $v^*=p^*+D$. Write $v^*=v^{*,+}+v^{*,-}$, where $v^{*,+}$ and $v^{*,-}$ are respectively the positive and the negative parts of $v^*$. We first claim that $v^*$ is an effective integral point, i.e., $v^{*,-}=0$. For the sake of a contradiction, let us assume the contrary, i.e., assume that $||v^{*,-}||_{\ell_1} >0$. Since $-D+v^{*,+}+v^{*,-}=-D+v^* = p^*$ is contained in $\Sigma(L)$, and because $v^{*,-} \leq 0$, the point $p^{*,+}=-D+v^{*,+}$ has to be in $\Sigma(L)$. Also $||v^{*,+}||_{\ell_1} < ||v^*||_{\ell_1}$ (because $||v^*||_{\ell_1} = ||v^{*,+}||_{\ell_1}+||v^{*,-}||_{\ell_1}$ and $||v^{*,-}||_{\ell_1}>0$). We obtain $||D+p^{*,+}||_{\ell_1} =||v^{*,+}||_{\ell_1} < ||D+p^{*}||_{\ell_1}$, which is a contradiction by the choice of $p^*$.
\noindent Therefore, we have 
\begin{align*}
r(D)\:&\:=\min \{\:\deg(v)\ |\ |D-v|= \emptyset,\: v \geq 0\: \}-1 \\
&=\:\min \{\:\deg(v)\ |\ v-D \in \Sigma(L),\: v \geq 0\: \}-1 \,\,  \textrm{ (By the first part of Lemma~\ref{ranksig_lem})}\\
&=\:\min \{\:||v||_{\ell_1}\ |\ v-D \in \Sigma(L),\: v\geq 0\: \}-1 \\ 
&=\:\min \{\:||D+p||_{\ell_1}-1\ |\ p\in\Sigma(L) \ \textrm{and } D+p\geq 0\: \}\\
&=\:dist_{\ell_1}(-D,\Sigma(L))-1 \qquad \qquad \qquad \qquad \qquad \qquad \textrm{(By the above arguments)}. 
\end{align*}
\end{itemize}
\end{proof}

Lemma~\ref{ranksig_lem} shows the importance of understanding the geometry of the Sigma-Region for the study of the rank function. This will be our aim in the rest of this section and in Section~\ref{sec:voronoi}. But we need to introduce another definition before we proceed. Apparently, it is easier to work with a ``continuous'' and ``closed'' version of the Sigma-Region. 

\begin{defi}\rm $\Sigma^{\mathbb{R}}(L)$ is the set of points in $\mathbb{R}^n$ that are not dominated by any point in $L$.
 \begin{align*}
 \Sigma^{\mathbb{R}}(L)\:&=\:\Bigl\{\: p \:|\: p \in \mathbb{R}^{n+1} \:\textrm{and}\: p \nleq q, \: \forall q \in L\: \Bigr\}\\
 &=\mathbb R^{n+1} \setminus \bigcup_{p \in L} H_p^-.
 \end{align*}       
By  $\Sigma^{c}(L)$ we denote the topological closure of $\Sigma^{\mathbb{R}}(L)$ in $\mathbb{R}^{n+1}$.            
 \end{defi}
\begin{rema}\rm
 One advantage of this definition is that it can be used to define the same Riemann-Roch machinery for any full dimensional sub-lattice of $H_0$. Indeed for such a sub-lattice $L$,  it is quite straightforward to associate a real-valued rank function to any point of $\mathbb R^{n+1}$ (c.f. Lemma~\ref{ranksig_lem}). The main theorems of the paper can be proved in this more general setting. As all the examples of interest for us are integral lattices, we have restricted the presentation to sub-lattices of $A_n$.  
 \end{rema}

\subsection{Extremal Points of the Sigma-Region}
\label{sec:extremal}
 We say that a point $p \in \Sigma(L)$ is an {\it extremal} point if it is a local minimum of the degree function. In other words 
\begin{defi}
\rm The set of extremal points of $L$ denoted by $\Ext(L)$ is defined as follows:
\begin{center}
$\Ext(L) :=\{\nu \in \Sigma (L) \ |\ \deg(\nu) \leq \deg(q) \ \forall \ 	q \in N(\nu) \cap \Sigma(L) \}).$
\end{center}
Recall that for every point $D\in\mathbb Z^{n+1}$, $N(D)$ is the set of neighbours of $D$ in $\mathbb Z^{n+1}$, which consists of all the points of $\mathbb Z^{n+1}$ which have distance at most one to $D$ in $\ell_{\infty}$ norm.
\end{defi}

\noindent  We also define extremal points of  $\Sigma^{c}(L)$ as the set of points that are local minima of the degree function and denote it by $\Ext^c(L)$. Local minima here is understood with respect to the topology of $\mathbb R^{n+1}$: $x$ is a local minima if and only if there exists an open ball $B$ containing $x$ such that $x$ is the point of minimum degree in $B \cap \Sigma^c(L)$. The following theorem describes the Sigma-Region of $L$ in terms of its extremal points.

\begin{theo}
\label{dom_thm} Every point of the Sigma-Region dominates an extremal point. In other words, $\Sigma(L)= \cup_{\nu \in \Ext(L)} H^{+}_{\nu}(\mathbb Z)$. Recall that $H^{+}_{\nu}(\mathbb Z)$ is the set of integral points of the cone $H^{+}_{v}$.
\end{theo}

Indeed, we first prove the following continuous version of Theorem~\ref{dom_thm}.
\begin{theo}\label{cdom_thm} For any (integral) sub-lattice $L$ of $H_0$, we have $\Sigma^{c}(L)= \cup_{\nu\in \Ext^c(L)} H^{+}_{\nu}$.
\end{theo}
And Theorem~\ref{dom_thm} is derived as a consequence of Theorem~\ref{cdom_thm}. 
The proof of these two theorems are presented in Section~\ref{app:theo-dom}. The proof shows that every extremal point of $\Sigma^c(L)$ is an integral point and $\Sigma(L)=\Sigma_{\mathbb{Z}}^c(L)+(1,\dots,1)$, where $\Sigma_{\mathbb{Z}}^c(L)$ denotes the set of integral points of $\Sigma^c(L)$. We refer to Section~\ref{app:theo-dom} for more details. 
\begin{prop}\label{prop:twosigmas}
We have $\Sigma(L)=\Sigma_{\mathbb{Z}}^c(L)+(1,\dots,1)$ and $\Ext(L) = \Ext^c(L)+(1,\dots,1).$ In particular, $\pi_0(\Ext^c(L)) = \pi_0(\Ext(L))$.
\end{prop}

 The important point about Theorem~\ref{dom_thm} is that one can use it to express $r(D)$ in terms of the extremal points of $\Sigma(L)$. 
For an integral point $D=(d_0,\dots,d_n) \in \mathbb Z^{n+1}$, let us define $\deg^{+}(D) := \deg(D^+)=\sum_{i\: :\: 	d_i \geq 0}d_i$ and $\deg^{-}(D) := \deg(D^-)=\sum_{i\: :\:d_i \leq 0}d_i$. We have:

\begin{lemm}\label{rank_lem} 
For every integral point $D\in \mathbb Z^{n+1}$, 
$$r(D)\:=\:\min\:\{\:\deg^{+}(v+D)\: | \: v \in \emph{Ext}(L)\:\}-1\: .$$
\end{lemm}

\begin{proof}
First recall that 
\begin{align*}
 r(D)\:&\:=\:\min\{\:\deg(E)\: |\: |D-E|\:=\:\emptyset \ \ \:\textrm{and}\ \ \: E \geq 0\:\}-1 \\
 &\:=\:\min\{\:\deg(E)\:|\: E-D \in \Sigma(L)\: \textrm{and}\ \ \: E \geq 0\:\}-1 \ \ \ \textrm{(By Lemma~\ref{ranksig_lem})}.
\end{align*}
\noindent Let $E \geq 0$ and $p=E-D$ be a point in $\Sigma(L)$. By Theorem \ref{dom_thm}, we know that $p$ is a point in $\Sigma(L)$ if and only if $p=v+E'$ for some point $\nu$ in $\Ext(L)$ and $E'\geq 0$. So we can write $E\:=p+D\:=\:\nu+E'+D$ where $\nu\in \Ext(L)$ and $E'\geq0$.  Hence we have 
\begin{align*}
r(D)=\min\{\:\deg(\nu+E'+D)\: |\: \nu \in \Ext(L),\: E' \geq 0 \: \textrm{and}\: \nu+E'+D \geq 0\:\}-1.
\end{align*}

\noindent We now observe that for every $\nu\in\mathbb Z^{n+1}$, the integral point $E' \geq 0$ of minimum degree such that $E'+\nu+D \geq 0$ has degree exactly $\deg^{+}(-\nu-D)$.
We infer that
\begin{align*}
\deg(\nu+E'+D)\:&\:=\:\deg(E')+\deg(\nu+D)\:=\:\deg^{+}(-\nu-D)+\deg(\nu+D)\\
&\:=\:\deg^{-}(\nu+D)+\deg(\nu+D)\:=\:\deg^{+}(\nu+D).
\end{align*}
We conclude that
$r(D)\:=\:\min \{\:\deg^{+}(v+D)\:|\: \nu \in \Ext(L)\:\}-1,$
and the lemma follows.
\end{proof}

\subsection{Min- and Max-Genus of Sub-Lattices of $A_n$ and Uniform Lattices}\label{sec:min-max-genus}
 We define two notions of genus for full-rank sub-lattices of $A_n$ , {\it min- and max-genus}, in terms of the extremal points of the Sigma-Region of $L$.  (The same definition works for full-rank sub-lattices of $H_0$.)
\begin{defi}[Min- and Max-Genus]
\rm The min- and max-genus of a given sub-lattice $L$ of $A_n$ of dimension $n$, denoted respectively by  $g_{min}$ and $g_{max}$, are defined as follows:
\begin{align*} 
&g_{min}(L)\:=\:\inf\:\{\:-\deg(\nu)\:|\:\nu \in \Ext(L)\:\}+1\:.\ \ \ \  \\ 
&g_{max}(L)\:=\:\sup\{\:-\deg(\nu)\:|\:\nu \in \Ext(L)\:\}+1\:.
\end{align*}
\end{defi}

\begin{rema} \rm There are some other notions of genus associated to a given lattice, e.g., the notion {\it spinor genus} for lattices developed by Eichler (see \cite{Eic52} and \cite{CS98}) in the context of integral quadratic forms. Every sub-lattice of $A_n$ provides a quadratic form in a natural way. But  a priori there is no relation between these notions. 
\end{rema}

\noindent It is clear by definition that $g_{min}\leq g_{max}$. But generally these two numbers could be different.
\begin{defi}\rm A sub-lattice $L \subseteq A_n$ of dimension $n$ is called {\it uniform} if $g_{min}=g_{max}$. The genus of a uniform sub-lattice is $g = g_{min} =g_{max}$.
\end{defi}
 \noindent As we will show later in Section~\ref{sec:examples}, sub-lattices generated by Laplacian of graphs are uniform.

\section{Proofs of Theorem~\ref{dom_thm} and Theorem~\ref{cdom_thm}}\label{app:theo-dom}

 In this section, we present the proofs of Theorem~\ref{dom_thm} and Theorem~\ref{cdom_thm}. This section is quite independent of the rest of this paper and can be skipped in the first reading. 

 Recall that $\Sigma^{\mathbb{R}}(L)$ is the set of points in $\mathbb{R}^{n+1}$ that are not dominated by any point in $L$ and $\Sigma^{c}(L)$ is the topological closure of $\Sigma^{\mathbb{R}}(L)$ in $\mathbb{R}^{n+1}$. Also, recall that $\Ext^c(L)$ denotes the set of extremal points of  $\Sigma^{c}(L)$. These are the set of points which are local minima of the degree function.
 As we said before, instead of working with the Sigma-Region directly, we initially work with $\Sigma^c(L)$. 
We first prove Theorem~\ref{cdom_thm}. Namely, we prove $\Sigma^{c}(L)= \cup_{\nu\in \Ext^c(L)} H^{+}_\nu$.
To prepare for the proof of this theorem, we need a series of lemmas.

\noindent The following lemma provides a description of $\Sigma^c(L)$ in terms of the domination order in $\mathbb R^{n+1}$. Recall that for two points $x=(x_0,\dots,x_n)$ and $y=(y_0,\dots,y_n)$, $x\leq y$ (resp. $x<y$) if $x_i \leq y_i$ (resp. $x_i<y_i$) for all $0\le i\le n$.

\begin{lemm} \label{sigc} $\Sigma^c(L)\:=\:\{\:p\:|\: p \in \mathbb{R}^{n+1}\: \textrm{and}\: \:\: \forall \: q \in L\::\: p \nless q \:\}.$
\end{lemm}

\begin{proof} Easy and omitted.
 \end{proof}

\begin{lemm} Extremal points of $\Sigma^c(L)$ are contained in $\partial(\Sigma^{c}(L))$. 
\end{lemm}

\begin{proof} Easy and omitted.
\end{proof}

Let $p$ be a point in $\Sigma^c(L)$ and let $d$ be a vector in $\mathbb{R}^{n+1}$. We say that $d$ is \emph{feasible} for $p$, if it satisfies the following properties:\\ 
1. $\deg(d)<0$.\\
2. There exists a $\delta_0(p,d)>0$ such that for every $0 \leq \delta \leq \delta_0(p,d)$, $p+\delta d \in \Sigma^c(L)$. By Lemma~\ref{sigc}, this means that $p+\delta d \nless p'$ for all lattice points $p'\in L$.\\

\noindent Furthermore, we define the function $\epsilon_{p,d}\::\:L\rightarrow \mathbb{R} \cup \{\infty\}$ as follows: 
\begin{align*}
\epsilon_{p,d}(q)\:=&\:\inf\:\{\:\epsilon\:|\: \epsilon>0\: \textrm{and}\: p+\epsilon d<q\:\}. 
\end{align*}
Let $I\:=\:\{\:i\:|\: 0 \leq i \leq n \:\:\textrm{and} \:\: p_i \geq q_i \:\}.$ We have the following explicit description of $\epsilon_{p,d}$.

\begin{equation}
\label{eps_eq}
\epsilon_{p,d}(q)=
\begin{cases}
   0 &\text{if $I = \emptyset$.}\\
   \max_{i\in I}\frac{(q_i-p_i)}{d_i} & \text{if $I\neq \emptyset$, $\forall \:\: i\in I, \:\: d_i<0$, and $\exists \:\:  \epsilon>0$ such that $p+\epsilon d<q$,}\\
  \infty & \text{otherwise.}

\end{cases}
\end{equation}
One can easily verify that

\begin{lemm} \label{rel_lem} For a point $p$ in $\Sigma^{c}(L)$,  $\epsilon_{d,p}(q)\: \geq \:\epsilon_{d^{-},p}(q)$ for all $q\in L\:.$ In the only cases when the inequality is strict, we must have $\epsilon_{d,p}(q) =\infty$ and $\epsilon_{d^-,p}(q)>0$.
\end{lemm}

\noindent We now prove the following lemma which links the function $\epsilon_{d,p}$ to the feasibility of $d$ at $p$.

\begin{lemm}\label{feas_lem} For a point $p$ in  $\Sigma^c(L)$ and $d$ in $\mathbb{R}^{n+1}$ with $\deg(d)<0$, $d$ is not feasible for $p$ if and only if 
 $\epsilon_{p,d}(q)=0$ for some $q\in L$.
\end{lemm}
\begin{proof} Let $p$ be a point of $\Sigma^c(L)$.

\noindent ($\Rightarrow$). Assume the contrary, then we should have the following properties:
\begin{enumerate}
\item \label{deg_prop1} $\deg(d)\:<\:0\:,$
\item \label{zero_prop1} $\epsilon_{p,d}(q)>0$ for all $q \in L\:,$
\end{enumerate}
We claim that $\inf_{q\in L}\:\{\:\epsilon_{p,d}(q)\:\}\:>\:\delta_0\:,$ for some $\delta_0\:>\:0\:.$
\noindent By the definition of $\epsilon_{p,d}$, if $\epsilon_{p,d}(q)\neq 0$, then $\epsilon_{p,d}(q)$ is at least $\min_{\{i:\:d_i<0\}} \frac{\{p_i\}}{|d_i|}$, where $0<\{p_i\}=p_i-\lceil p_i-1\rceil\leq 1$ is the rational part of $p_i$ if $p_i$ is not integral, and is $1$ if $p_i$ is integral. As the number of indices is finite, we conclude that $\delta_0 =\min_{\{i:\:d_i<0\}} |\frac{\{p_i\}}{d_i}|$ and the claim holds.
\noindent
It follows that $p+\epsilon d\nless q$ for all $q$ in $L$ and for all $0 \leq \epsilon \leq \delta_0 $.  This implies that $d$ is feasible for $p$.
\newline

\noindent($\Leftarrow$). 
If $\epsilon_{p,d}(q)=0$ for some $q\in L$, then there exists a $\delta_0>0$ such that $p+\delta d < p'$ for every $0<\delta \leq \delta_0$. This shows that $d$ is not feasible for $p$. 
\end{proof}

\begin{coro} \label{ext_cor} For a point $p$ in  $\Sigma^c(L)$, $p$ is an extremal point if and only if for every vector $d\in \mathbb{R}^{n+1}$ with $\deg(d)<0$, we have $\epsilon_{p,d}(q)=0$ for some $q$ in $L$.
\end{coro}

 Combining Lemma~\ref{rel_lem} and Corollary~\ref{ext_cor}, we obtain the following result:
\begin{lemm} \label{dom_lem} If $p$ is not an extremal point of  $\Sigma^c(L)$, then there exists a vector $d$ in $H^{-}_O$ which is feasible for $p$.
\end{lemm}
\begin{proof} 
If $p$ is not an extremal point of $\Sigma^c(L)$, then there exists a vector $d_0$ in $\mathbb{R}^{n+1}$ that is feasible for $p$.  By Corollary \ref{ext_cor}, $d_0$ has the following properties:
\begin{enumerate}
\item \label{deg_prop2} $\deg(d_0)<0\:,$
\item \label{zero_prop2} $\epsilon_{d_0,p}(q)>0$ for all $q\in L\:,$
\end{enumerate}
Let $d:=d_0^-$. We have $\deg(d)<0$, since $\deg(d_0)<0$ and $d=d_0^-$. By Lemma \ref{rel_lem}, we have  $\epsilon_{d_0,p}(q)\:\geq\:\epsilon_{d,p}(q)$
for all $q\in L$, and in the only cases for $q$ when the inequality is strict we have $\epsilon_{d,p}(q)>0$. We infer that $d$ also satisfies Properties \ref{deg_prop2} and \ref{zero_prop2}. By Corollary~\ref{ext_cor}, $d$ is also feasible for $p$ and by construction, $d$ belongs to  $H^{-}_O$; the lemma follows.  
\end{proof}

Consider the set $\deg(\Sigma^c(L))=\{\:\deg(p)\:|\: p \in \Sigma^c(L)\:\}.$ The next lemma shows that the degree function is bounded below on the elements of $\Sigma^c(L)$ (by some negative real number). 

\begin{lemm} \label{sup_lem} For an $n-$dimensional sub-lattice $L$ of $A_n$, $\inf(\deg(\Sigma^c(L))$ is finite. 
\end{lemm}
\begin{proof} It is possible to give a direct proof of this lemma. But using our results in Section~\ref{sec:voronoi} allows us to shorten the proof. So we postpone the proof to Section~\ref{sec:voronoi}.
\end{proof}

\noindent We are now in a position to present the proofs of Theorem~\ref{cdom_thm} and Theorem~\ref{dom_thm}.

\begin{proof}[Proof of Theorem~\ref{cdom_thm}] Consider a point $p$ in $\Sigma^{c}(L)$. We should prove the existence of an extremal point $\nu\in \Ext^c(L)$ such that $\nu \leq p$. 

\noindent Consider the cone $H^-_p$. As a consequence of Lemma~\ref{sup_lem}, we infer that the region $\Sigma^c(L)\cap H^-_p$ is a bounded closed subspace of $\mathbb R^{n+1}$, and so it is compact. The degree function $\deg$ restricted to this compact set, achieves its minimum on some point $\nu\in \Sigma^c(L)\cap H^-_p$. We claim that $\nu\in \Ext^c(L)$. Suppose that this is not the case. By Lemma~\ref{dom_lem}, there exists a feasible vector $d\in H^-_O$ for $\nu$, i.e., such that $\nu+\delta d\in \Sigma^c(L)$ for all sufficiently small $\delta>0$. Now it is easy to check that 
\begin{itemize}
\item $\nu+\delta\:d\in H^-_p$ and hence $\nu+\delta d \leq p\:,$
\item $\deg(\nu+\delta\:d)< \deg(\nu)$.
\end{itemize}
This contradicts the choice of $\nu$.
\end{proof}

\begin{proof}[Proof of Theorem \ref{dom_thm}]

In order to establish Theorem~\ref{dom_thm}, we first prove that every point in $\Ext^c(L)$ is an integral point. For the sake of a contradiction, suppose that there exists a non integral point in $\Ext^c(L)$. Let $p=(p_0,\dots,p_n)$ be such a point and suppose without loss of generality that $p_0$ is not integer. We claim that the vector $d=-e_0=(-1,0,0,\dots,0)$ is feasible. Indeed it is easy to check that $\epsilon_{p,d}(q) > 0$ for all $q\in L$, and so by Corollary~\ref{ext_cor} we conclude that $p$ could not be an extremal point of $\Sigma^c(L)$.

\noindent Let $\Sigma^c_{\mathbb{Z}}(L)$ be the set of integral points of $\Sigma^c(L)$. We show that $\Sigma^c_{\mathbb{Z}}(L) + (1,\dots,1) = \Sigma(L)$. Note that as soon as this is proved, Theorem~\ref{cdom_thm} and the fact that extremal points of $\Sigma^c(L)$ are all integral points implies Theorem~\ref{dom_thm}. 
\newline

\noindent We prove $\Sigma^c_{\mathbb{Z}}(L) + (1,\dots,1) \subseteq \Sigma(L)$.--- Let $u=v+(1,\dots,1) \in \Sigma^c_{\mathbb{Z}}(L) + (1,\dots,1)$, for a point $v\in\Sigma^c_{\mathbb{Z}}(L)$. To show $u\in \Sigma(L)$ we should prove that $\forall q\in L:\: u\nleq q$. Suppose that this is not the case and let $q\in L$ be such that $u \leq q$. It follows that $u-(1,\dots,1) < q$ and hence, $v \notin \Sigma^c(L)$, which is a contradiction.
\newline

\noindent We prove $\Sigma(L) \subseteq \Sigma^c_{\mathbb{Z}}(L) + (1,\dots,1)$.--- A point $u$ in $\partial\Sigma^c(L)$ is contained in $H^-(q)$ for some $q$ in $L$ and hence $u \leq q$. We infer that $\Sigma(L)$ is contained in the interior of $\Sigma^c(L)$, and so for each point $p$ of $\Sigma(L)$, every vector in $\mathbb R^{n+1}$ of negative degree will be feasible. 
By Lemma~\ref{sup_lem}, there exists a point $p_c \in \partial\Sigma^C(L)$ such that $p=p_c+t(1,\dots,1)$ for some $t>0$. It follows that $p>p_c$. By Theorem \ref{cdom_thm}, $p_c \in H^+_\nu$ for some $\nu$ in $\Ext^c(L)$.  This implies that  $p>\nu$ for some $\nu \in \Ext^c(L)$. By definition, $p$ is an integral point and we just showed that $\nu$ is also an integral point. Hence we can further deduce that $p \geq \nu+(1,\dots,1)$.  We infer that $p-(1,\dots,1)\geq \nu$ and therefore, $p-(1,\dots,1)\in \Sigma^c(L)$ (because $H_\nu^+ \subset \Sigma^c(L)$). It follows that $p\in \Sigma^c_{\mathbb{Z}}(L) + (1,\dots,1)$. 

The proof of Theorem~\ref{dom_thm} is now complete.
\end{proof}

%%%%%%%%%%%%%%%%%%%%%%%%%%%% VORONOI

\section{Voronoi Diagrams of Lattices under Simplicial Distance Functions} 
\label{sec:voronoi}

In this section, we provide some basic properties of the Voronoi diagram of a sub-lattice $L$ of $A_n$ under a simplicial distance function $d_{\triangle}(.\:,.)$ which we define below. The distance function $d_{\triangle}(.\:,.)$ has the following explicit form, and as we will see in this section, is the distance function having the homotheties of the standard simplex in $H_0$ as its balls (which explains the name simplicial distance function).  For two points $p$ and $q$ in $H_0$, the simplicial distance between $p$ and $q$ is defined as follows
$$d_{\triangle}(p,q) :=  \inf \Bigl\{\lambda\,|\, q-p +\lambda (1,\dots,1) \geq 0 \Bigr\}.$$
The basic properties of $d_\triangle$ are better explained in the more general context of polyhedral distance functions that we now explain.

\subsection{Polyhedral Distance Functions and their Voronoi Diagrams} Let $Q$ be a convex polytope in $\mathbb R^n$ with the reference point $O=(0,\dots,0)$ in its interior. The {\it polyhedral distance function} $d_Q(.\:,.)$ between the points of $\mathbb R^n$ is defined as follows:
\[\forall\: p,q\in \mathbb R^n,\: d_{Q}(p,q)\::=\:\inf\{\lambda\geq 0\:|\:q \in p +\lambda.Q\}, \: \textrm{where}\ \ \:\lambda.Q\:=\:\{\:\lambda.x\:|\:x\in Q\:\}.\]
$d_Q$ is not generally symmetric, indeed it is easy to check that $d_Q(.\:,.)$ is symmetric if and only if the polyhedron $Q$ is centrally symmetric i.e., $Q=-Q$. Nevertheless $d_{Q}(.\:,.)$ satisfies the triangle inequality.

\begin{lemm}\label{lin_lem}
For every three points $p,q,r\in \mathbb R^{n}$, we have $d_{Q}(p,q)+d_Q(q,r) \geq d_{Q}(p,r)$.
 In addition, if $q$ is a convex combination of $p$ and $r$, then $d_Q(p,q)+d_Q(q,r)=d_Q(p,r)$.
\end{lemm}

\begin{proof}
To prove the triangle inequality, it will be sufficient to show that if $q\in p+\lambda.Q$ and $r\in q+\mu.Q$, then $r\in p+(\lambda+\mu).Q$. We write $q=p+\lambda.q'$ and $r=q+\mu.r'$ for two points $q'$ and $r'$ in $Q$. We can then write $r=p+\lambda.q'+\mu.r'=p+(\lambda+\mu)(\frac{\lambda}{\lambda+\mu}.q'+\frac{\mu}{\lambda+\mu}.r')$. $Q$ being convex and $\lambda,\mu\geq0$, we infer that $\frac{\lambda}{\lambda+\mu}.q'+\frac{\mu}{\lambda+\mu}.r'\in Q$, and so $r\in p+(\lambda+\mu).Q$. The triangle inequality follows.

\noindent To prove the second part of the lemma, let $t \in [0,1]$ be such that $q=t.p+(1-t).r\:.$ By the triangle inequality, it will be enough to prove that $d_Q(p,q)+d_Q(q,r)\leq d_Q(p,r)$. Let $d_Q(p,r)=\lambda$ so that $r = p+\lambda.r'$ for some point $r'$ in $Q$. We infer first that 
$q=t.p+(1-t).r=t.p+(1-t)(p+\lambda.r')= p+(1-t)\lambda.r'$, which implies that $d_Q(p,q)\leq(1-t)\lambda$. Similarly we have $t.r=t.p+t\lambda.r'=q-(1-t)r+t\lambda.r'$. It follows that $r=q+t\lambda r'$ and so $d_Q(q,r)\leq t\lambda\:.$ We conclude that $d_Q(p,q)+d_Q(q,r)\leq d_Q(p,r)$, and the lemma follows.
\end{proof}

We also observe that the polyhedral metric $d_Q(.\:,.)$ is translation invariant, i.e.,
\begin{lemm}\label{lem:trans-inv} For any two points $p,q$ in $\mathbb R^n$, and for any vector $v\in\mathbb R^n$, we have $d_Q(p,q) = d_Q(p-v,q-v)$. In particular, $d_{Q}(p,q)=d_{Q}(p-q,O)=d_{Q}(O,q-p)$. 
\end{lemm}
\begin{proof} The proof is easy: if $q\in p+\lambda.Q$, then $q-v \in p-v+\lambda.Q$, and vice versa. 
\end{proof}

\begin{rema}\rm
The notion of a polyhedral distance function is essentially the concept of a gauge function of a convex body that has been studied in \cite{Siegel89}. Lemmas \ref{lin_lem} and \ref{lem:trans-inv} can be derived in a straight forward way from the results in \cite{Siegel89}. For the sake of easy reference, we included them here.
\end{rema}

\noindent Consider a discrete subset $\mathcal S$ in $\mathbb{R}^{n}$. For a point $s$ in $\mathcal S$, we define the {\it Voronoi cell} of $s$ with respect to $d_{Q}$  as $V_Q(s)\:=\:\{\: p \in \mathbb{R}^{n}\:|\:d_{Q}(p,s)\leq d_{Q}(p,s')\:\: \textrm{for any other point}\:\: s'\in \mathcal S\:\}\:.$

\noindent The {\it Voronoi diagram} $\Vor_Q(\mathcal S)$  is the decomposition of $\mathbb{R}^n$ induced by the cells $V_{Q}(s)$, for $s\in\mathcal S\:.$ We note however that this need not be a cell decomposition in the usual sense.

We state the following lemma on the shape of cells $V_Q(s)$.

\begin{lemm}\label{lem:star}\cite{CD85} Let $\mathcal S$ be a discrete subset of $\mathbb R^n$ and $\Vor_Q(\mathcal S)$ be the Voronoi cell decomposition of $\mathbb R^n$. For any point $s$ in $\mathcal S$, the Voronoi cell $V_Q(s)$ is a star-shaped polyhedron with $s$ as a kernel. 
\end{lemm}

\begin{proof}
It is easy to see that $V_Q(s)$ is a polyhedron. We show that it is star-shaped. Assume the contrary. Then there is a line segment $[s,r]$ and a point $q$ between $s$ and $r$ such that $r \in V_Q(s)$ and $q\notin V_Q(s)$. Suppose that $q$ is contained in $V(s')$ for some $s' \neq s$. We should then have $d_{Q}(q,s) > d_{Q}(q,s')$. By Lemma~\ref{lin_lem}, $d_{Q}(r,s)=d_{Q}(r,q)+d_{Q}(q,s)$. We infer that
\begin{align*} 
d_Q(r,s)\:&=\: d_{Q}(r,q)+d_{Q}(q,s) > d_{Q}(r,q)+d_{Q}(q,s')\ge\: d_{Q}(r,s'),\: \textrm{contradicting} \:r\in V_Q(s).
 \end{align*} 
\end{proof}

\subsection{Voronoi Diagram of Sub-Lattices of $A_n$}\label{sec:voronoi-sub-lattice}
  Voronoi diagrams of root lattices under the Euclidean metric have been studied previously in literature. Conway and Sloane \cite{CS91,CS98}, describe the Voronoi cell structure of root lattices and their duals under the Euclidean metric. 

\noindent Here we study Voronoi diagrams of sub-lattices of $A_n$ under polyhedral distance functions (and later under the simplicial distance functions $d_{\triangle}(.\:,.)$). We will see the importance of this study in the proof of Riemann-Roch Theorem in Section~\ref{sec:R-R-theo}, and in the geometric study of the Laplacian of graphs in Section~\ref{sec:examples}. 

\noindent Let $L$ be a sub-lattice of $A_n$ of full rank. Note that $L$ is a discrete subset of the hyperplane $H_0$ and $H_0\simeq \mathbb R^n$. Let $Q\subset H_0$ be a convex polytope of dimension $n$ in $H_0$.  We will be interested in the Voronoi cell decomposition of the hyperplane $H_0$ under the distance function $d_Q(.\:,.)$ induced by  the points of $L$. The following lemma, which essentially uses the translation-invariance of $d_Q(.\:,.)$, shows that these cells are all simply translations of each other.
\begin{lemm} \label{trans_lem} 
For a point $p$ in $L, V_Q(p) = V_Q(O)+p\:.$ As a consequence, $\textrm{\emph{Vor}}_{Q}(L)=V_Q(O)+L$.
\end{lemm}

\begin{proof}Easy and omitted.
\end{proof}

\noindent By Lemma~\ref{trans_lem}, to understand the Voronoi cell decomposition of $H_0$, it will be enough to understand the cell $V_{Q}(O)$. We already know that $V_Q(O)$ is a star-shaped polyhedron. The following lemma shows that $V_Q(O)$ is compact, and so it is a (non-necessarily convex) star-shaped polytope.

\begin{lemm} \label{compact_lem} The Voronoi cell $V_{Q}(O)$ is compact.
\end{lemm}

\begin{proof}
 The proof is standard. It will be sufficient to prove that $V_Q(O)$ does not contain any infinite ray. Indeed, $V_Q(O)$ being star-shaped and closed, this  will imply that $V_Q(O)$ is bounded and so we have the compactness. 

Assume, for the sake of a contradiction, that there exists a vector $v\neq O$ in $H_0$ such that the ray $t.v$ for $t\geq 0$ is contained in $V_{Q}(O)$. This means that 
\begin{align}
\label{keypro} \textrm{For every}\: t\geq 0\: \textrm{and for every} \: p\in L,\: \textrm{we have} \: 	d_{Q}(t.v,O) \leq d_{Q}(t.v,p).
\end{align} 
Choose a real number $\lambda$ such that $0<\lambda<d_Q(v,O)$. By Lemma~\ref{lin_lem}, $d_Q(t.v,O)=td_Q(v,O)>\lambda t$ for $t>0$.   By the definition of $d_Q$, the choice of $\lambda$ and Property~(\ref{keypro}), the polytope $t.v+t\lambda.Q=t.(v+\lambda Q)$ does not contain any point $p\in L$ for $t>0$. Let $\mathcal C=\bigcup_{t\geq 0} t.(v+\lambda.Q)$. It is easy to check that $\mathcal C$  is the cone generated by $v+\lambda.Q$. It follows that $\mathcal C $ does not contain any lattice point apart from $O$ (for $t=0$).  In addition, $Q$ being a polytope of dimension $n$, $\mathcal C$ should be a cone of full dimension in $H_0$. But this will provide a contradiction, because as we will show below for any vector $\bar v$ with rational coordinates in $H_0$, the open ray $t.\bar v$ for $\ t > 0$ contains a lattice point in $L$. (And it is clear that any cone $\mathcal C$ of full dimension in $H_0$ contains a rational vector.) To see this, observe that a basis for $L$ is also a basis for the $n$-dimensional $\mathbb Q$-vector space $H_0(\mathbb Q)$. Here $H_0(\mathbb Q)$ denotes the rational points of the hyperplane $H_0$. This means that $\bar v$ can be written as a rational combination of some points in $L$. Multiplying by a sufficiently large integer number $N$, $N.\bar v$ can be written as an integral combination of the same points in $L$, i.e.,  $N.\bar v \in L$, and this finishes the proof of the lemma. 
\end{proof}

From now on, we will restrict ourselves to two special polytopes $\triangle$ and $\bar\triangle$ in $H_0$. They are both standard simplices of $H_0$ under an appropriate isometry  $H_0\simeq\mathbb R^n$.  
\noindent The $n$-dimensional regular simplex $\triangle(O)$ centred at the origin $O$ has vertices at the points $b_0,b_1,\dots,b_n$.  For all $\: 0\leq i,j\leq n$, the coordinates of $b_i$ are given by: 
\begin{center}
$(b_i)_j =
\begin{cases}
n & \text{if i=j,}   \\
-1 & \text{otherwise.}
\end{cases}
$
\end{center}

The simplex $\bar{\triangle}(O)$ is the {\it opposite} simplex to $\triangle(O)$, i.e.,  $\bar{\triangle}(O):=-\triangle(O)$. 
The simplicial distance functions $d_\triangle(.\:,.)$ and $d_{\bar\triangle}(.\:,.)$ are the distance functions in $\mathbb R^{n+1}$ defined by $\triangle$ and $\bar\triangle$ respectively.
It is easy to check the following anti-symmetric property for the above distance functions: {\it For any pair of points $p,q\in\mathbb R^{n+1}$, we have $d_\triangle(p,q)=d_{\bar\triangle}(q,p)$.}  (This is indeed true for any convex polytope $Q$: $d_Q(p,q)=d_{\bar Q}(q,p)$, where $\bar Q=-Q$.)

\paragraph*{Notation.} In the following we will use the following terminology: For a point $v\in H_0$, we let $\triangle(v) = v+\triangle(O)$ and $\bar\triangle(v)=v+\bar\triangle(O)$. More generally given a real $\lambda\geq 0$ and $v\in H_0$, we define $\triangle_\lambda(v) = v+\lambda\triangle(O)$, and similarly, $\bar\triangle_\lambda(v) = v+\lambda\bar\triangle(O)$. We can think of these as {\it balls of radius} $\lambda$ around $v$ for $d_\triangle$ and $d_{\bar\triangle}$ respectively.

\vspace{.5cm}

\begin{figure}[!htb]
\begin{center}
\includegraphics[width=0.25\linewidth]{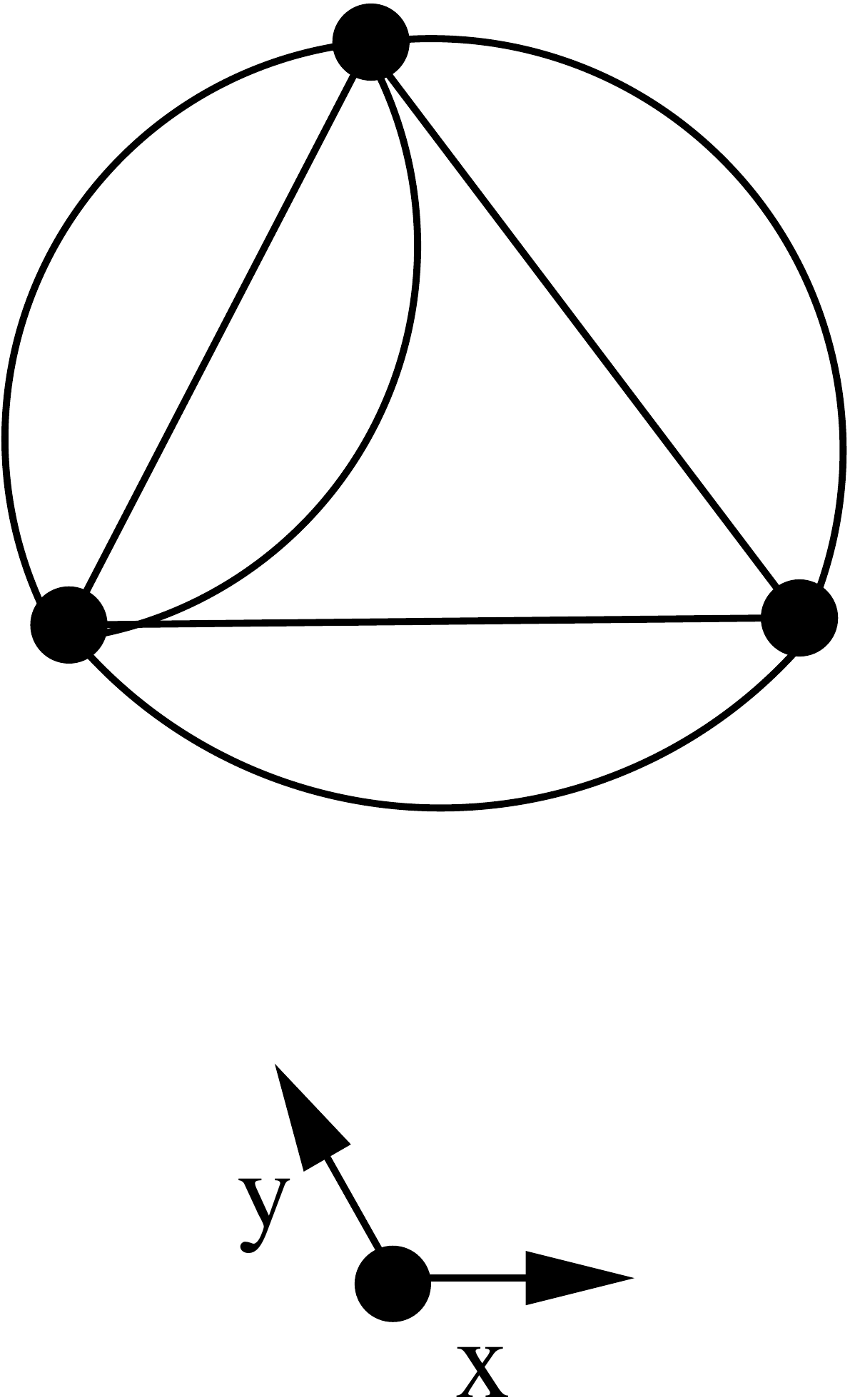}
\hspace{3cm}
\includegraphics[width=0.45\linewidth]{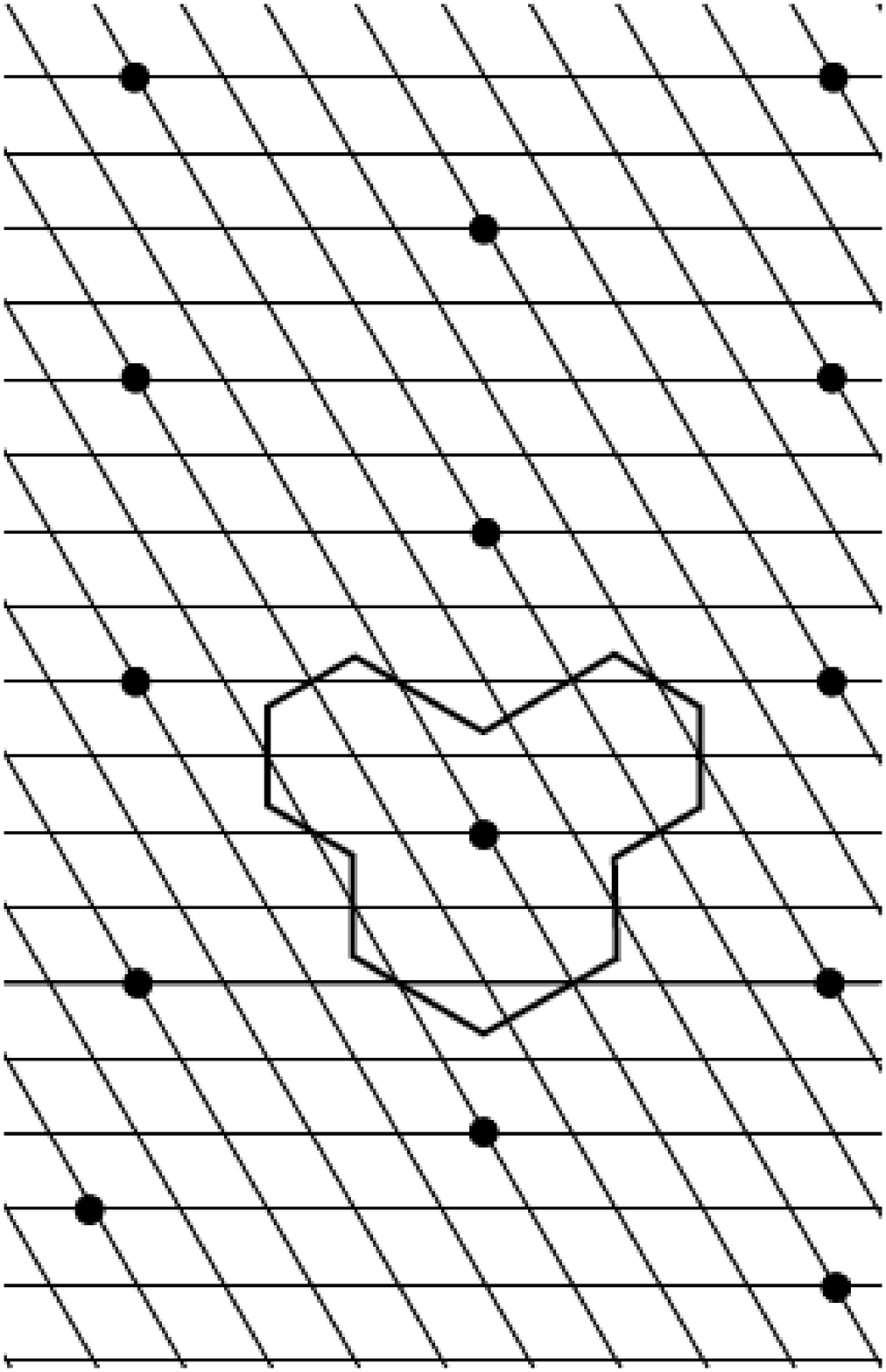}
\caption{The shape of a Voronoi-cell in the Laplacian lattice of a graph with three vertices. The multi-graph $G$ has three vertices and $7$ edges. The lattice $A_2$ is generated by the two vectors $x = (1,-1,0)$ and $y=(-1,0,1)$. The  corresponding Laplacian sub-lattice of $A_2$, whose elements are denoted by $\bullet$, is generated by the vectors $(-5,3,2)=-3x+2y$ and $(3,-5,2) = 5x+2y$ (and $(2,2,-4) = -2x-4y$), which correspond to the vertices of $G$.}
\label{defVor}
\end{center}
\end{figure}

\noindent The following lemma shows that the definition given in the beginning of this section coincides with the definition of $d_\triangle$ given above. We can explicitly write a formula for $d_\triangle(.\:,.)$ and $d_{\bar\triangle}(.\:,.)$ in the hyperplane $H_0$:

\begin{lemm}\label{d_form} 
For two points $p=(p_0,p_1,\dots,p_n)$ and $q=(q_0,q_1,\dots,q_n)$ in $H_0$, the $\triangle$-simplicial distance from $p$ to $q$ is given by 
$\: d_\triangle(p,q)\:=\: |\:\bigoplus_{i=0}^n(q_i-p_i)\:|.$
And the $\bar\triangle$-simplicial distance from $p$ to $q$ is given by 
$\: d_{\bar\triangle}(p,q)\:=\: |\:\bigoplus_{i=0}^n(p_i-q_i)\:|.$
\noindent Here the sum $\bigoplus_i (x_i-y_i) $ denotes the tropical sum of the numbers $x_i-y_i$.
\end{lemm}

\begin{proof} 
By the anti-symmetry property of the distance function $d_{\triangle}(.\:,.)$ (namely $d_{\triangle}(p,q) = d_{\bar\triangle}(q,p),\; \forall p,q$), we only need to prove the lemma for $d_{\triangle}(.\:,.)$. By definition, $d_\triangle(p,q)$ is the smallest positive real $\lambda$ such that $q\in p+\lambda.\triangle$. The simplex $\triangle$ being the convex hull of the vectors $b_i$ defined above, it follows that for an element $x\in \lambda.\triangle$, there should exist non-negative reals $\mu_i\geq0$ such that $\sum_{i=0}^n\mu_i=\lambda$ and $x=\mu_0 b_0+\mu_1b_1+\dots+\mu_nb_n$. From the definition of the vector $b_i$'s, we obtain
$x=(n+1)(\mu_0,\mu_1,\dots,\mu_n)-\lambda(1,\dots,1)$. It follows that $d_{\triangle}(p,q)$ is the smallest $\lambda$ such that $q-p+\lambda.(1,\dots,1)$ becomes equal to $(n+1)(\mu_0,\mu_1,\dots,\mu_n)$ for some $\mu_i\geq0$ such that $\sum_i\mu_i=\lambda$. Let $\lambda_0$ be the smallest positive real number such that the vector $\mu:=\frac1{n+1}(q-p+\lambda_0.(1,\dots,1))$ has non-negative coordinates. As $p,q\in H_0$, a simple calculation shows that the other condition $\sum_i\mu_i=\lambda_0$ holds automatically, and hence  such $\lambda_0$ is equal to $d_{\triangle}(p,q)$. It is now easy to see that $\lambda_0 = \max_i\: (p_i-q_i) = -\min_{i}\:(q_i-p_i)$. It follows that $d_{\triangle}(p,q)=|\bigoplus_{i=0}^n(q_i-p_i)|.$ 
 \end{proof}

\subsection{Vertices of $\mathrm {Vor}_{\triangle}(L)$ that are Critical Points of a Distance Function.}\label{sec:criticalvertices}
For a discrete subset $\mathcal S$ of $H_0$ (e.g., $\mathcal S=L$), the simplicial distance function $h_{\triangle,\mathcal S}:H_0 \rightarrow \mathbb{R}$ is defined as follows:
\begin{align*}
 h_{\triangle,\mathcal S}(x) \:&=\: \bigoplus_{p\in \mathcal S} d_{\triangle}(x,p)\:=\:\min_{p\in \mathcal S} d_{\triangle}(x,p).
\end{align*}
By  definition, it is straightforward to verify that
$h_{\triangle,\mathcal S}(x)\:=\:\sup \{\:\lambda\ |\ (x+\lambda.\triangle) \cap \mathcal S=\emptyset\:\}$.

\noindent Note that our definition above exactly imitates  the classical definition of a distance function~\cite{GJ08}. 
 In what follows, we restrict ourselves to $\mathcal S=L$. 
 %But it is quite straightforward to generalise what follows below to other discrete subsets of $H_0$.
   
Let $L$ be a full-rank sub-lattice of $A_n$ and $h_{\triangle,L}$ be the distance function defined by $L$. We first give a description of $\partial\Sigma^c(L)$ (see Section~\ref{sec:extremal}) in terms of $h_{\triangle,L}$.  The {\it lower-graph} of $h_{\triangle,L}$ is the graph of the function $h_{\triangle,L}$ in the negative half-space of  $\mathbb R^{n+1}$, i.e., in the half-space of $\mathbb R^{n+1}$ consisting of points  of negative degree. More precisely, the {\it lower-graph} of $h_{\triangle,L}$, denoted by $\Gr(h_{\triangle,L})$, consists of all the points $y-h_{\triangle,L}(y)(1,\dots,1)$ for $y\in H_0$.  (In the example given in Figure~\ref{defVor}, these are all the vertices of the polygon drawn in the plane $H_2$ (the right figure) having one concave and one convex neighbours on the polygon. There are six of them.)

We have
\begin{lemm}\label{lem:lower-graph} The lower-graph of $h_{\triangle,L}$ and $\partial\Sigma^c(L)$ coincide, i.e., $\emph{Gr}(h_{\triangle,L})=\partial\Sigma^c(L)$.   
\end{lemm} 

 In order to present the proof of Lemma~\ref{lem:lower-graph}, we need to make some remarks. Let $p$ be a point of $L$. The function $f_p:H_0 \rightarrow \mathbb R^{n+1}$ is defined as follows: 
\[\forall \: y\in H_0,\: f_p(y) := \sup\:\{y_t\:|\:y_t=y-t.(1,\dots,1), t\geq 0, \:\textrm{and}\: y_t\leq p\:\}.\]
\noindent Note that $\sup$ is defined with respect to the ordering of $\mathbb R^{n+1}$, and is well-defined because $y_t\geq y_{t'}$ if and only if $t\leq t'$. Remark also that $f_p(y)$ is finite.

\begin{rema}\rm
The above notion has the following tropical meaning:  Let $\lambda_p=\min\:\{t\in\mathbb R\:\:|\:\:t\odot p \oplus y = y\}$. Then $y_p = (-\lambda_p)\odot y$. The numbers $\lambda_p$ are used in~\cite{DS04} to define the tropical closest point projection into some tropical polytopes. For a finite set of points $p_1,\dots,p_l$ with the tropical convex-hull polytope $Q$, the tropical projection map $\pi_Q$ at the point $y$ is defined as $\pi_Q(y) = \lambda_{p_1}\odot p_1 \oplus \dots \oplus \lambda_{p_l}\odot p_l.$ It would be interesting to explore the connection between the work presented here and the theory of tropical polytopes. 
\end{rema}

 A simple calculation shows that $f_p(y) = y-|\bigoplus_i (p_i-y_i)|.(1,\dots,1)$, and hence by Lemma~\ref{d_form}, we obtain $f_p(y)=y-d_{\triangle}(y,p).(1,\dots,1)$. In other words, $f_p(y)$ is the lower-graph of the function $d_{\triangle}(.\:,p)$. We claim that for all $y\in H_0,\:y-h_{\triangle,L}(y)(1,\dots,1)\:=\:\sup_{p\in L}f_p(y).$ Here, $\sup$ is understood as before with respect to the ordering of $\mathbb R^{n+1}$. In other words, the lower-graph $\Gr(h_{\triangle,L})$ is the lower envelope of the graphs $\Gr(f_p)$ for $p\in L$. To see this, remark that $\sup_{p\in L}f_p(y) = \sup_{p\in L} (y-d_{\triangle}(y,p).(1,\dots,1)) = y-(\min_{p\in L}d_{\triangle}(y,p)).(1,\dots,1) = y-h_{\triangle,L}(y).(1,\dots,1)$.

\begin{proof}[Proof of Lemma~\ref{lem:lower-graph}] It is easy to see that for every point $y\in H_0$, the intersection of the half-ray $\{\:y-t(1,\dots,1)\:|\: t\geq0\:\}$ with $\partial \Sigma^c(L)$ is the point $y-h_{\triangle,L}(L).(1,\dots,1)\in\Gr(h_{\triangle,L})$. This gives the lemma. More precisely, by the definition of $\Sigma^c(L)$ (see Section~\ref{sec:extremal}), we have
\begin{align*}
\partial\Sigma^c(L) \:&=\: \{\:z\:|\:z\leq p\:\textrm{for some}\: p\in L\:\textrm{and}\:z\nless p,\:\forall p\in L \}\\
&\hspace{-0.1cm}\:=\:\{\:\sup_{p\in L}f_p(y)\:|\:y\in H_0\:\}\:=\:\Gr(h_{\triangle,L})\:\:\:(\textrm{By the discussion above)}.
\end{align*}
\end{proof}

\noindent It is possible to strengthen Lemma~\ref{lem:lower-graph} and to obtain a description of the Voronoi diagram $\Vor_{\triangle}(L)$ in terms of the boundary of the Sigma-Region. The following lemma can be seen as the simplicial Voronoi diagram analogue of the classical result that {\it the Voronoi diagram under the Euclidean metric is the projection of a lower envelope of paraboloids \cite{Edels01}.} 

\begin{lemm}\label{lowenv_lem} The Voronoi diagram of $L$ under the simplicial distance function $d_{\triangle}(.\:,.)$ is the projection of $\partial \Sigma^{c}(L)$ along $(1,\dots,1)$ onto the hyperplane $H_0$. More precisely, for any $p \in L$, the Voronoi cell $V_{\triangle}(p)$ is obtained as the image of  $H_p^-\cap \partial\Sigma^c(L)$ under the projection map $\pi_0$.
\end{lemm}

\begin{proof}[Proof of Lemma~\ref{lowenv_lem}] By definition, $H_p^-$ consists of the points which are dominated by $p$. It follows that the intersection $H_p^-\cap \partial\Sigma^c(L)$ consists of all the points of $\partial\Sigma^c(L)$ which are dominated by $p$. By Lemma~\ref{lem:lower-graph}, the boundary of $\Sigma^c(L)$, $\partial\Sigma^c(L)$ coincides with the graph of the simplicial distance function $h_{\triangle,L}$. It follows that  the intersection $H_p^-\cap \partial\Sigma^c(L)$ consists of all the points of the lower-graph of $h_{\triangle,L}$ that are dominated by $p$. By definition, any point of the lower-graph of $h_{\triangle,L}$ is of the form $y-h_{\triangle,L}(y).(1,\dots,1)$ for some $y\in H_0$. By definition of the function $f_p$, such a point is dominated by $p$ if and only if $h_{\triangle,L}(y) \geq f_p(y)$. By definition, we know that $h_{\triangle,L}(y)\leq f_p(y)$ for all $y\in H_0$. We infer that for $y\in H_0$, $y- h_{\triangle,L}(y).(1,\dots,1)
 \in H_p^-\cap \partial\Sigma^c(L)$ if and only if $h_{\triangle,L}(y)= f_p(y)$, or equivalently, if and only if $y\in V_\triangle(p)$. We conclude that $V_{\triangle}(p) = \pi_0(H_p^-\cap \partial\Sigma^c(L))$ and the lemma follows.
 \end{proof}

As we show in the next two lemmas, it is possible to describe Voronoi {\it vertices} that are local maxima of $h_{\triangle,L}$ as the projection of the extremal points of the Sigma-Region onto the hyperplane $H_0$ (see below, Lemma~\ref{ver_cor}, for a precise statement).
  
\noindent Let us denote by $\Crit(L)$ the set of all local maxima of $h_{\triangle,P}$. We have
 \begin{lemm}\label{crit_lem}
The critical points of $L$ are the projection  of the extremal points of $\Sigma^c(L)$ along the vector $(1,\dots,1)$. In other words, $\emph{Crit}(L)= \pi_0(\emph{Ext}^c(L))$. 
\end{lemm}
\begin{proof} Let $c$ be a point in $\Crit(L)$, and let $x=c-h_{\triangle,L}(c).(1,\dots,1)$, be the corresponding point of the lower-graph of $h_{\triangle,L}$, $\Gr(h_{\triangle,L})\:(=\partial\Sigma^c(L)$ by Lemma~\ref{lem:lower-graph}). Note that $\pi_0(x)=c$. We claim that $x\in \Ext^c(L)$. Assume the contrary. Then there should exist an infinite sequence $\{\:x_i\:\}_{i=1}^{\infty}$ such that $(i)\:x_i\in \partial\Sigma^c(L)$, $(ii)\:\deg(x_i)<\deg(x)$, and $(iii)\:\lim_{i\rightarrow \infty} \: x_i\:=\:x$. By $(i)$ and Lemma~\ref{lem:lower-graph}, we can write $x_i=p_i-h_{\triangle,L}(p_i).(1,\dots,1)$ for some $p_i\in H_0$. By $(ii)$, we should have $-(n+1)h_{\triangle,L}(p_i)=\deg(x_i)<\deg(x)=-(n+1)h_{\triangle,L}(c)$ for every $i$, and so $h_{\triangle,L}(p_i)>h_{\triangle,L}(p_i)$. By $(iii)$, we have $\lim_{i\rightarrow \infty}\: p_i=c$. All together, we have obtained an infinite sequence of points $\{p_i\}$ in $H_0$ such that $h_{\triangle,L}(p_i)>h_{\triangle,L}(c)$ and $\lim_{i\rightarrow \infty}\: p_i=c$. This is a contradiction to our assumption that $c\in \Crit(L)$ is a local maximum of $h_{\triangle,L}$.
 A similar argument shows that for every point $x\in\Ext^c(L)$, $\pi_0(x)$ is in $\Crit(L)$, and  the lemma  follows.
\end{proof}

By Proposition~\ref{prop:twosigmas}, we have $ \pi_0(\Ext^c(L)) = \pi_0(\Ext(L))$, and so 
\begin{coro}\label{cor:crit}
We have $\emph{Crit}(L)= \pi_0(\Ext(L))$.
\end{coro}

The following lemma gives a precise meaning to our claim that the critical points are the Voronoi vertices of the Voronoi diagram, and will be used in Section~\ref{sec:examples} in the proof of Theorem~\ref{thm:main-laplacian} (also used to drive Theorem~\ref{theo:crit-n!}). 
\begin{lemm}\label{ver_cor}
Each $v\in \emph{Crit}(L)$ is a vertex of the Voronoi diagram $\emph{Vor}_{\triangle}(L)$: there exist $n+1$ different points $p_0,\dots,p_n$ in $L$ such that $v \in \bigcap_i V(p_i)$. More precisely, a point $v\in H_0$ is critical, i.e.,  $v\in \emph{Crit}(L)$, if and only if it satisfies the following property: for each of the $n+1$ facets $F_i$ of $\bar\triangle_{h_{\triangle,L}(v)}(v)$, there exists a point $p_i \in L$ such that $p_i \in F_i$ and $p_i$ is not in any of $F_j$ for $j \neq i$. 
\end{lemm}
Remark that this shows that every point in $\Crit(L)$ is a vertex of the Voronoi diagram $\Vor_\triangle(L)$.
\begin{proof}
  We first prove that for every $v \in \Crit(L)$, there exist $(n+1)$ different points $p_i\in L, \:i=0,\dots,n$, such that the corresponding Voronoi cells $V_{\triangle}(p_i)$ shares $v$, i.e., such that $v\in V_{\triangle}(p_i)$ for $i\in\{\:0,\dots,n\:\}$. By Lemma~\ref{crit_lem}, we know that there exists a point $x\in \Ext^{c}(L)$ such that $\pi_0(x)=v$. 
 We will prove the following: there exist $(n+1)$ different points $p_i\in L, \:i=0,\dots,n$ such that $x\in H^-_{p_i}$ for all $i\in\{\:0,\dots,n\:\}$. Once this has been proved, we will be done. Indeed by Lemma~\ref{lowenv_lem}, we know that that every Voronoi cell $V_{\triangle}(p)$, for $p\in L$, is of the form $\pi_0(H_p^-)\cap \partial\Sigma^c(L)$. So $v \in \pi_0 (H_{p_i}^-\cap \partial\Sigma^c(L))=V_{\triangle}(p_i)$ for each point $p_i$, and this is exactly what we wanted to prove.

 To prove the second part, it will be enough to show that the points $p_i$ have the desired property. Remark that we have $d_{\bar\triangle}(p_i,v)=d_\triangle(v,p_i)=h_{\triangle,L}(v)$, so $p_i \in \partial\bar\triangle_{h_{\triangle,L}(v)}(v)$ for all $i$. By the choice of $p_i$, we have $(p_i)_j > x_j$ for all $j \neq i$ and $(p_i)_i = x_i$. Since $v=\pi_0(x)$, it is now easy to see that $p_i$ is in the facet $F_i$ of $\bar\triangle_{h_{\triangle,L}(v)}(v)$ defined by
 $$F_i = \{\:u\in\bar\triangle_{h_{\triangle,L}(v)}(v)\:\:|\:\: u_i =v_i - h_{\triangle,L}(v) \:\:\textrm{and}\:\: u_j \geq v_j - h_{\triangle,L}(v)\:\}.$$
(Remark that $d_{\bar\triangle}(x,v) = |\oplus_j (x_j - v_j)|$ so this is a facet of $\bar\triangle_{h_{\triangle,L}(v)}(v)$.)
And $p_i$ is not in any of the other facets $F_j$ (since $(p_i)_j > v_j - h_{\triangle,L}(v)$ for $j \neq i$). So the proof of one direction  is now complete.
To prove the other direction, let $v$ be a point such that each of the $n+1$ facets $F_i$ of $\bar\triangle_{h_{\triangle,L}(v)}(v)$ has a point $p_i \in L$ and $p_i$ is not in any of the other facets $F_j$ for $j \neq i$. We show that $v$ is critical, i.e., $v$ is a local maxima of $h_{\triangle,L}$. It will be enough to show that for any non-zero vector $d\in H_0$ of sufficiently small norm, there exists one of the points $p_i$ such that $d_{\triangle}(v+d,p_i) < h_{\triangle,L}(v)=d_{\triangle}(v,p_i)$. For all $j$, by the characterisation of the facet $F_j$ (see above) and by $p_j \notin F_k$ for all $k\neq i$, we have $d_{\triangle}(v+d,p_j) = d_{\bar\triangle}(p_j,v+d) = |\bigoplus_k (p_j)_k-v_k-d_k| = d_j+v_j-(p_j)_j=h_{\triangle,L}(v)+d_j$ if all $d_k$'s are sufficiently small (namely if for all $k$, $|d_k| \leq \epsilon$ where $\epsilon >0$ is chosen so that $2\epsilon < \min_{j, k: k\neq j} \bigl[(p_j)_k-v_k + h_{\triangle,L}(v)\bigr]$).  As $d \in H_0$ and $d\neq
 0$, there exists $i$ such that $d_i < 0$. It follows that $h_{\triangle,L}(d+v) \leq d_{\triangle}(v+d,p_i) < h_{\triangle,L}(v)$. And this shows that $v$ is a local maximum of $h_{\triangle,L}$. The proof of the lemma is now complete.
\end{proof}

\subsection{Proof of Lemma~\ref{sup_lem}}
We end this section by providing the promised short proof of Lemma~\ref{sup_lem}, which claims that the degree function is bounded below in the region $\Sigma^c(L)$.
\begin{proof}[Proof of Lemma~\ref{sup_lem}]
In Section~\ref{sec:criticalvertices} we obtained  the following explicit formula for $f_p(y)$: 
$$\forall y\in H_0,\: f_p(y)\:=\:y-d_\triangle(y,p)(1,\dots,1).$$
 We infer that  
\begin{equation}\label{eq:vp}
\forall \: y\in V_{\triangle}(p):\:\:\:f_p(y)\:=\:y-h_{\triangle,L}(y).(1,\dots,1). 
\end{equation} 

 By Lemma \ref{lem:lower-graph}, we have $\partial\Sigma^c(L)=\Gr(h_{\triangle,L})$. It follows from Equation~\ref{eq:vp} that 
\begin{align*}
&\:\:\:\:\:\:\:\partial\Sigma^{c}(L)\:=\:\{\:f_p(y)\:|\:y \in V_\triangle(p)\:\textrm{and}\: p \in L\}. 
\end{align*} 

 We now observe that:
\begin{align*}
\forall\: y\in H_0:\:\deg(f_p(y))\:=\:\deg(y)-(n+1)d_{\triangle}(y,p)\:=&\:-(n+1)d_{\triangle}(y,p).
\end{align*}
\noindent This shows that $\deg(f_p(y))$ depends only on the simplicial distance $d_\triangle$ between $y$ and $p$.
 By translation invariance of the simplicial distance function (Lemma~\ref{lem:trans-inv}), translation invariance of the Voronoi cells (Lemma~\ref{trans_lem}), and the above observations, we obtain
\begin{align*}
\inf(\deg(\Sigma^c(L)))\:&=\:\inf_{y\in V_{\triangle}(p)}\{\:-(n+1)d_{\triangle}(y,p)\:\}\\ 
&=\:\inf_{y\in V_{\triangle}(O)}\{\:-(n+1)d_{\triangle}(y,O)\:\}\\ 
&=\:-(n+1)\sup_{y\in V_{\triangle}(O)}\{\:d_{\triangle}(y,O)\:\}.
\end{align*}

\noindent By Lemma \ref{compact_lem}, we know that $V_{\triangle}(O)$ is compact. Also the function $d_{\triangle}(O,y)$ is continuous on $y$.
Hence $\sup_{y \in V_{\triangle}(O)}\{d_{\triangle}(y,O)\}\}$ is finite and the lemma follows.
\end{proof}

%%%%%%%%%%%%%%%%%%%%%%%%%%%%%%%%%% MAIN THEOREM

\section{Riemann-Roch Theorem for Uniform Reflection Invariant Sub-Lattices}
\label{sec:R-R-theo}

Consider a full dimensional sub-lattice $L$ of $A_n$ and its Voronoi diagram $\Vor_{\triangle}(L)$ under the simplicial distance function. From the previous sections, we know that the points of $\Crit(L)$ are vertices of $\Vor_{\triangle}(L)$.  We know that $V_{\triangle}(O)$ is a compact star-shaped polyhedron with $O$ as a kernel, and that the other cells are all translations of $V_{\triangle}(0)$ by points in $L$. Consider now the subset $\Crit V_\triangle(O)$ of vertices of $V_\triangle(O)$ which are in $\Crit(L)$.    The sub-lattices of $A_n$ of interest for us should have the following symmetry property:  

\begin{defi}[Reflection Invariance]\rm
A sub-lattice $L\subseteq A_n$ is called {\it reflection invariant} if $-\Crit(L)$ is a translate of $\Crit(L)$, i.e., if there exists $t \in \mathbb{R}^{n+1}$ such that $-\Crit(L)=\Crit(L)+t$. $L$ is called {\it strongly reflection invariant} if the same property holds for $\Crit V_{\triangle}(O)$, i.e., if there exists $t \in \mathbb{R}^{n+1}$ such that $-\Crit V_\triangle(O)=\Crit V_\triangle(O)+t$.
\end{defi}
By translation invariance, it is easy to show that every strongly reflection invariant sub-lattice of $A_n$ is indeed reflection invariant. 

\noindent Also, note that the vector $t$ in the definition of reflection invariance lattices above is not uniquely defined: by translation invariance, if $t'$ is linearly equivalent to $t$, $t'$ also satisfies the property given in the definition.

\paragraph*{Reflection Invariance and Involution of $\Ext(L)$.}
Let  $L$ be reflection invariant and $t \in \mathbb{R}^{n+1}$ be a point such that $-\Crit(L)=\Crit(L)+t$. This means that for any $c\in \Crit(L)$ there exists a unique $\bar c \in \Crit(L)$ such that $c+\bar c=-t$.
 By Lemma~\ref{crit_lem} and Corollary~\ref{cor:crit}, for every point $c$ in $\Crit(L)$, there exists a point $\nu$ in $\Ext(L)$ such that $c=\pi_0(\nu)$. Thus, for every point $\nu$ in $\Ext(L)$, there exists a point $\bar{\nu}$ in $\Ext(L)$ such that $\pi_0(\nu+\bar{\nu})=-t$. This allows to define an involution $\phi (= \phi_t) : \Ext(L) \rightarrow \Ext(L)$:
 
 \begin{center} 
 For any point $\nu \in \Ext(L)$, $\phi(\nu) := \bar \nu$. 
 \end{center}
Note  that $\phi$ is well defined. Indeed, if there exist two different points $\bar \nu_1$ and $\bar \nu_2$ such that $\pi_0(\nu+\bar{\nu}_i)=-t$ for $i=1,2$, then  $\pi_0(\bar{\nu_1}) =\pi_0(\bar{\nu_2})$ and this would imply that $\bar{\nu_1} > \bar{\nu_2}$ or $\bar{\nu_2} > \bar{\nu_1}$ which contradicts the hypothesis that $\bar{\nu_1},\bar{\nu_2} \in \Ext(L)$. A similar argument shows that $\phi$ is a bijection on $\Ext(L)$ and is an involution.

\subsection{A Riemann-Roch Inequality for Reflection Invariant Sub-Lattices: Proof of Theorem~\ref{rrineq_theo}}

 In this subsection, we provide the proof of the Riemann-Roch inequality stated in Theorem~\ref{rrineq_theo} for reflection invariant sub-lattices of $A_n$. We refer to Section~\ref{sec:min-max-genus} for the definition of $g_{min}$ and $g_{max}$. 
 
 Let $L$ be a reflection invariant sub-lattice of $A_n$. We have to show the existence of  a canonical point $K\in\mathbb Z^{n+1}$ such that for every point $D \in \mathbb{Z}^{n+1}$, we have
\begin{equation}\label{Rie-Ro_ineq}3g_{min}-2g_{max}-1\: \leq\: r(K-D)-r(D)+\deg(D)\: \leq\:g_{max}-1\:	.
\end{equation}
$K$ is defined up to linear equivalence (which is manifested  in the choice of $t$ in the definition of reflection invariance).

\hspace{.5cm}

\noindent{\bf Construction of a Canonical Point $K$.}

\noindent We define the canonical point $K$ as follows: Let $\nu_0\in \Ext(L)$ be an extremal point such that $\nu_0+\phi(\nu_0)$ has the maximum degree, i.e., $\nu_0=\mathrm{argmax}\:\{\:\deg(\nu+\phi(\nu))\:|\: \nu \in \Ext(L)\:\}$. The map $\phi$ is the involution defined above. 
Define $K:=-\nu_0-\phi(\nu_0)$.

 \vspace{.3cm}

\noindent{\bf Proof of the Riemann-Roch Inequality.}
We first observe that $K$ is well-defined and for any point $\nu$ in $\Ext(L)$, $\nu\:+\:\bar{\nu}\: \leq\: -K.$ This is true because all the points $\nu+\bar \nu$ are on the line $-t + \alpha(1,\dots,1)), \ \alpha\in \mathbb R$, and $K$ is chosen in such a way to ensure that $-K$ has the maximum degree among the points of that line. We infer that for any point $\nu\in \Ext(L)$, there exists an effective point $E_\nu$ such that $\nu+\bar \nu =-K-E_\nu$. Using this, we first derive an upper bound on the  quantity $\deg^{+}(K-D+\bar{\nu})-\deg^{+}(\nu+D)$ as follows:
\begin{align}
\deg^{+}(K-D+\bar{\nu}) - \deg^{+}(\nu+D)&=\label{eq1}\deg^{+}(-\nu-\bar{\nu}-E_\nu-D+\bar{\nu})-\deg^{+}(\nu+D)\\
&= \label{ineq2} \deg^{+}(-\nu-E_\nu-D)-\deg^{+}(\nu+D)\\ 
&\leq \deg^{+}(-\nu-D)-\deg^{+}(\nu+D)\\ 
&= \deg(-\nu-D)\:=\:-\deg(\nu)-\deg(D)\\
&\label{lem_gmax}\leq g_{max}-\deg(D)-1.
\end{align}

\noindent To obtain Inequality (\ref{ineq2}), we use the fact that if $E \geq 0$ then $\deg^{+}(D-E) \leq \deg^{+}(D)$. Also  remark that Inequality (\ref{lem_gmax}) is a simple consequence of the definition of $g_{max}$.

\noindent Now, we obtain a lower bound on the quantity $\deg^{+}(K-D+\bar{\nu})-\deg^{+}(\nu+D)$.  In order to do so, we first obtain an upper bound on the degree of $E_\nu$, for the effective point $E_\nu$ such that $\nu+\bar \nu =-K-E_\nu$. To do so, we note that by the definition of $K$ and by the definition of $g_{min}$, we have $\deg(K)=\min(\deg(-\nu-\bar \nu)) \geq 2 g_{min}-2$. Also observe that by the definition of $g_{max}$, we have $\deg(-\nu-\bar \nu) \leq 2g_{max}-2$. It follows that
\[\deg(E_v) = -\deg(K)+\deg(-\nu-\bar \nu) \leq 2(g_{max}-g_{min}).\]
 \noindent We proceed as follows
\begin{align*}
\deg^{+}(K-D+\bar{\nu}) - \deg^{+}(\nu+D)\:&=\: \deg^{+}(-\nu-E_\nu-D)-\deg^{+}(\nu+D)\\
\:&\ge\: \deg^{+}(-\nu-D)-\deg(E_\nu)-\deg^{+}(\nu+D)\\
\:&\ge\: 2(g_{min}-g_{max})+\deg^{+}(-\nu-D)-\deg^{+}(\nu+D)\\
\:&\ge\:2(g_{min}-g_{max})-\deg(\nu+D)\\
\:&=\:2(g_{min}-g_{max})-\deg(\nu)-\deg(D)\\
\:&\ge\:3g_{min}-2g_{max}-\deg(D)-1. 
\end{align*} 
\noindent The last inequality follows from the definition of $g_{min}$.
Now since the map $\phi(\nu)=\bar{\nu}$ is a bijection from $\Ext(L)$ onto itself, we can easily see that
\begin{align*}
3g_{min}-2g_{max}-\deg(D)-1 &\leq \min_{\nu \in \Ext(L)}\deg^{+}(K+\bar{\nu}-D) -\min_{\nu \in \Ext(L)}\deg^{+}(\nu+D)\\
&\le g_{max}-\deg(D)-1.
\end{align*}

\noindent By Lemma \ref{rank_lem} and the fact $\phi$ is a bijection, we know that:
\begin{center} 
$r(D)=\min_{\nu \in \Ext(L)}\deg^{+}(\nu+D)-1,$\\
$ r(K-D)=\min_{\bar \nu \in \Ext(L)}\deg^{+}(K-D+\bar \nu)-1.$
\end{center}

\noindent Finally we infer that $3g_{min}-2g_{\max}-\deg(D)-1\:\leq\:r(K-D)-r(D)\:\leq\:g_{max}-\deg(D)-1\:,$
and the Riemann-Roch Inequality (\ref{rrineq_theo}) follows.

\begin{rema}\label{rema:R-R-digraph}\rm
As the above proof shows, we indeed obtain a slightly stronger inequality 
\[g_{min}-\deg(D)-1-\max_{\nu\in \textrm{Ext}(L)} \deg(E_\nu)\:\leq\:r(K-D)-r(D).\]
In particular if $E_\nu=0$ for all $\nu\in \Ext(L)$ (see Section~\ref{sec:digraphs} for examples, e.g., regular digraphs), we have:
\[g_{min}-\deg(D)-1\:\leq\:r(K-D)-r(D)\leq g_{max}-\deg(D)-1.\]
\end{rema}

 We remark that the proof technique used above is quite similar to the one used by Baker and Norine~\cite{BaNo07}.

 \begin{rema}\rm From Lemma \ref{rank_lem}, it is easy to obtain the inequality $\deg(D)-r(D) \leq g_{max}$, for all sub-lattices $L$ of $A_n$ and all $D\in\mathbb Z^{n+1	}$. This inequality is usually referred to as Riemann's inequality. Note that the Riemann-Roch inequality (\ref{Rie-Ro_ineq}) is more sensitive on (and contains more information about) the extent of ``un-evenness'' of the extremal points, while the above trivial inequality does not provide any such information.
 \end{rema}

\subsection{Riemann-Roch Theorem for Uniform Reflection Invariant Lattices}

 Recall that a lattice $L$ is called uniform if $g_{max}=g_{min}$, i.e., every point in $\Ext(L)$ has the same degree. By Corollary~\ref{cor:crit} and the definition of $h_{\triangle}$, this is equivalent to saying that the set of critical values of $h_{\triangle,L}$ is a singleton. We call $g=g_{max}=g_{min}$ the {\it genus} of the lattice.

The following is a direct consequence of Theorem~\ref{rrineq_theo}. However we give it as a separate theorem.
\begin{theo}\label{theo:R-R-uniform} Every uniform reflection invariant sub-lattice $L\subseteq A_n$ of dimension $n$ has the Riemann-Roch property.
\end{theo}

\begin{proof}
Let $D\in\mathbb Z^{n+1}$. If $L$ is a reflection invariant lattice, we can apply Theorem \ref{rrineq_theo} to obtain $3g_{min}-2g_{max}-1 \leq r(K-D)-r(D)+\deg(D) \leq g_{max}-1,$
where $K$ is the canonical point defined as in the proof of Theorem~\ref{rrineq_theo}. Since $L$ is uniform we have $g_{max}=g_{min}=g$ and we obtain $r(K-D)-r(D)+\deg(D)=g-1.$ It remains to show that $\deg(K)=2g-2$. But, we know from the construction of $K$ that $K=-(\nu+\bar{\nu})$ for a point $\nu \in\Ext(L)$. Since $L$ is uniform, we infer that $\deg(K)=-\deg(\nu)-\deg(\bar \nu)=2g-2$ (and also that $K=-\nu-\bar \nu$, $\forall\:\nu\in\Ext(L)$). 
\end{proof}

We say that a sub-lattice $L$ of $A_n$ has a Riemann-Roch formula if there exists an integer $m$ 
and an integral point $K_m$, or simply $K$, of degree $2m-2$ (a canonical point) such that for
every integral point $D$, we have:
$$r(D)-r(K-D)=\deg(D)-(m-1).$$

The following result shows the amount of geometric information one can obtain from the Riemann-Roch Property. 

\begin{theo}\label{theo:R-R-class} A sub-lattice $L$ has a Riemann-Roch formula if
and only if it is uniform and reflection invariant. Moreover, for a uniform and
reflection invariant lattice $m=g$ (the genus of the lattice). 
\end{theo}

The rest of this section is devoted to the proof of this theorem. One direction is already shown, we prove the other direction.

\noindent 

We first prove that  

\begin{claim}
If $L$ has a Riemann-Roch formula, then $m=g_{max}$.
\end{claim}
\begin{proof}
The Riemann-Roch formula for a point $D$ with $\deg(D)> 2m-2 $ implies that $\deg(D)-r(D)=m$. We know that if
$\deg(D)>2g_{max}-2$ then $\deg(D)-r(D)\leq g_{max}$. This for $D$ with $\deg(D) \geq 2\max\{m,g_{max}\}-2$ shows that $m\leq g_{max}$. 
By the Riemann-Roch formula, we have $r(D) \geq 0$ for any $D$ with $\deg(D) \geq m$. Let $D=-\nu_{max}$, where $\nu_{max}$ is an extremal point of minimal degree. Remark that we have $r(D)=-1$. This shows that $m\geq g_{max}$. And we infer that $m = g_{max}$.
\end{proof}

We now prove that 
\begin{claim}
If $L$ has a Riemann-Roch formula, then $L$ is uniform and $m=g$. 
\end{claim}
\begin{proof}
 Let $N$ be the set of points of $\Sigma(L)$ of degree $-g_{max}+1$. We note that every point in $N$ is extremal, i.e., $N \subset \Ext(L)$. To prove the uniformity, we should prove that $N=\Ext(L)$. We claim that $\Sigma(L)=\cup_{\nu \in N} H^{+}_\nu$, and this in turn implies that $N =\Ext(L)$. Indeed, if the claim holds, then every extremal point $\nu \in \Ext(L)$ should dominate a point $u$ in $N$, and so $u=\nu$, meaning that $N = \Ext(L)$. 

To prove the claim, we proceed as follows.  Let $-D$ be a point in $\Sigma(L)$. We know that $r(D) = -1$. We should prove the existence of a point $\nu$ in $N$ such that $\nu \leq -D$. By the Riemann-Roch formula there exists $E \geq 0$ with $\deg(E)=g_{max}-1-\deg(D)$ and $r(D+E)=-1$. The point $-D-E$ has degree $-g_{max}+1$ and so is in $N$. In addition $-D - E \leq -D$. And this is what we wanted to prove. The proof of the uniformity is now complete.
\end{proof}

To finish the proof of the theorem, it remains to show that 

\begin{claim}\label{prop:R-R-inverse} If a uniform sub-lattice $L$ of $A_n$ of full dimension has a Riemann-Roch formula, then it is reflection invariant.
\end{claim}

\begin{proof}
Consider a uniform lattice satisfying the Riemann-Roch property. By Lemma~\ref{ranksig_lem}, we know that for a point $\nu$ in $\Ext(L)$, $r(-\nu)=-1$. Now, if we evaluate the Riemann-Roch formula for $D=-\nu$, we get $r(-\nu)=r(K+\nu)$. Hence, we have $r(-\nu)=r(K+\nu)=-1$. Again by Lemma~\ref{ranksig_lem}, this implies that $-K-\nu$ is a point in $\Sigma(L)$. By the Riemann-Roch property and Claim 2 above, $\deg(K)=2g-2$. Since $L$ is uniform and $\nu\in\Ext(L)$, we have $\deg(\nu)=g-1$. We infer that $\deg(K+\nu)=g-1$, and it follows that $-K-\nu$ is an extremal point of $\Sigma(L)$. We now define $\bar{\nu}=-K-\nu$. Clearly, the map $\nu \rightarrow \bar{\nu}$ is a bijection from $\Ext(L)$ onto itself. 
Let $t=\pi_0(-K)$. We obtain $t=\pi_0(-\nu-\bar{\nu})=-\pi_0(\nu)-\pi_0(\bar{\nu})$ for all $\nu\in\Ext(L)$. By Corollary~\ref{cor:crit}, we have $\Crit(L)=\pi_0(\Ext(L))$ and hence $t=-c-\bar{c}$ for every $c$ in $\Crit(L)$. This implies that $-\bar{c}=t+c$. To finish the proof, observe that $\bar{c}\rightarrow c$ is a bijection from $\Crit(L)$ onto itself, and so we have $-\Crit(L)=\Crit(L)+t$.
\end{proof}

The proof of Theorem~\ref{theo:R-R-class} is now complete.

%%%%%%%%%%%%%%%%%%%%%%%%%%%%% Examples

\section{Examples}
\label{sec:examples}
In this section we study the machinery we presented in the previous sections through a few classes of examples. 
\subsection{Lattices Generated by Laplacian of Connected Graphs}
\label{sec:laplacian}
 Probably the most interesting examples of the sub-lattices of $A_n$ are generated by Laplacian of connected multi-graphs (and more generally directed multi-graphs) on $n+1$ vertices. 
In this subsection, we provide a geometric study of these sub-lattices.  We prove the following result:

\begin{theo}\label{theo:laplacian-unif+reflect}\label{thm:graphic} For any connected graph $G$, the sub-lattice $L_G$ of $A_n$ generated by the Laplacian of $G$ is  strongly reflection invariant and uniform.  
\end{theo}

\noindent Theorem~\ref{thm:graphic} will be a direct consequence of Theorem~\ref{thm:main-laplacian} below. Combining this theorem with Theorem~\ref{rr_theo} gives the main result of~\cite{BaNo07}.
\begin{coro}[Theorem 1.12 in~\cite{BaNo07}]\label{coro:R-R-graph} For any connected graph $G$ on $n+1$ vertices and with $m$ edges, the Laplacian lattice $L_G$ has the Riemann-Roch property. In addition, we have $g_{\max}=g_{\min}=m-n$ and the canonical point $K$ is given by $(\delta_0-2,\delta_1-2,\dots,\delta_n-2)$ of $\mathbb Z^{n+1}$ where $\delta_i$'s are the degrees of the vertices of $G$.
\end{coro}

\begin{rema}\rm Using reduced devisors, and the results of~\cite{BaNo07}, it is probably quite straightforward to obtain a proof of Theorem~\ref{thm:graphic}. (This is not surprising since, as we pointed out in the previous section, a lattice with a Riemann-Roch formula has to be uniform and reflection invariant.)  The proof we will present for Theorem~\ref{thm:graphic} gives indeed more than what is the content of this theorem. We  give a complete description of the Voronoi-diagram and  its dual Delaunay triangulation. And we do not use reduced divisors, which is the main tool used in the previous proofs of the Riemann-Roch theorem. As we will see, the form of the canonical divisor for a given graph (and the genus) as defined in~\cite{BaNo07} comes naturally out of this explicit description. \end{rema}

Let $G$ be a connected graph on $n+1$ vertices $v_0,v_1,\dots,v_n$ and $m$ edges. Let $L_G$, or simply $L$ if there is no risk of confusion, be the Laplacian sub-lattice of $A_n$. We summarise the main properties of the lattice $L_G$ and the matrix $Q$, defined in Section~\ref{sec:intro}. $L_G$ is an $n$-dimensional sub-lattice of $A_n$ with $\{b_0,\dots,b_{n-1}\}$ as a basis  such that the $(n+1) \times (n+1)$ matrix $Q$ has $\{b_0,\dots,b_{n-1}\}$ as the first $n$ rows and $b_n=-\sum_{i=0}^{n-1}b_{i}$ as the last row. 
In addition, the matrix 
\begin{equation}\label{gra_form}
Q=
\begin{bmatrix}
  \delta_0      & -b_{01} & -b_{02}  \hdots & -b_{0n} \\
 -b_{10}   &  \delta_1    & -b_{12}  \hdots & -b_{1n} \\
 \vdots    &  \vdots & \ddots\\
 -b_{n0}   &  b_{n1} & -b_{n2}  \hdots &  \delta_n
\end{bmatrix}
\end{equation}
has the following properties:
\begin{itemize}
\item[$(C_1)$] $b_{ij}$'s are integers, $b_{ij}\geq0$ for all $0\leq i \neq j\leq n$ and $b_{ij}=b_{ji},\:\:\forall i\neq j$.
\item[$(C_2)$] $\delta_i=\sum_{j=1,j\neq i}^{n}b_{ij}=\sum_{j=1,j\neq i}^{n}b_{ji}$ (and is the degree of the $i-$ vertex).
\end{itemize}
We denote by $B$ the basis $\{b_0,\dots,b_{n-1}\}$ of $L_G$.

\subsection*{Voronoi Diagram $\Vor_{\triangle}(L_G)$ and the Riemann-Roch Theorem for Graphs}
We first provide a decomposition of $H_0$ into simplices with vertices in $L$ such that the vertices of each simplex forms an affine basis of $L_G$. Recall that a subset of lattice points $X\subset L$ of size $n+1$ is called an affine basis of $L$, if for $v \in X$, the set of vectors $u-v$, $u\in X$ and $u\neq v$, forms a basis of $L$. In other words, if the simplex defined by $X$ is minimal (which means it is full-dimensional and has minimum volume among all the (full-dimensional) simplices whose vertices lie in $L$). The whole decomposition is derived from the symmetries of the affine basis $B$, and describes in a very nice way the Voronoi decomposition $\Vor_\triangle(L_G)$. What follows could be considered as an explicit construction of the ``{\it Delaunay dual}'', $\Del_{\triangle}(L_G)$, of $\Vor_{\triangle}(L_G)$.

We consider the family of total orders on the set $\{\:0,1,\dots,n\:\}$. A total order $<_{\pi}$ on  $\{\:0,1,\dots,n\:\}$ gives rise to an element $\pi$ of the symmetric group $S_{n+1}$, defined in such a way that $\pi(0) <_\pi \pi(1) <_\pi \dots <_\pi \pi(n-1)<_{\pi}\pi(n)$. It is clear that the set of all total orders on $\{0,\dots,n\}$ is in bijection with the elements of $S_{n+1}$. In addition the total orders which have $n$ as the maximum element are in bijection with the subgroup $S_n\subset S_{n+1}$ consisting of all the permutation which fix $n$, i.e., $\pi(n)=n$. 
In the following when we talk about a permutation in $S_n$, we mean a permutation of $S_{n+1}$ which fixes $n$. For $\pi\in S_n$, we denote by $\bar\pi$ the {\it opposite permutation} to $\pi$ defined as follows:  we set $\bar\pi(n)=n$ and $\bar\pi(i)=\pi(n-1-i)$ for all $i=0,\dots,n-1$. In other words, for all $i=0,\dots,n-1$, $i<_\pi j$ if and only if $j<_{\bar\pi} i$, and $j\leq_{\bar\pi}n$ for all $j$. Let $C_{n+1}$ denotes the group of cyclic permutations of $\{0,\dots,n\}$, i.e., $C_{n+1}=<\sigma>$ where $\sigma$ is the  element of $S_{n+1}$ defined by $\sigma(i) = i+1$ for $0\leq i\leq n-1$ and $\sigma(n)=0$. It is easy to check that $S_{n+1}=S_nC_{n+1}.$

Let $<_{\pi}$ be a total order such that $\pi\in S_n$, i.e., $\pi\in S_{n+1}$ and $\pi(n)=n$.
We first define a set of vectors $B^{\pi}=\{\:b^\pi_0,\dots,b^\pi_{n}\:\}$ as follows: 
\begin{align*}
\forall\ i \in \:\{\:0,\dots, n\:\},\: \ \  b^{\pi}_i\::=\: \sum_{j \leq_\pi i}b_j\:.
\end{align*}
In particular, note that $b^{\pi}_n=b^\pi_{\pi(n)}:= \sum_{ j\leq_{ \pi} \pi(n) }b_j\:=\:\sum_{j=0}^n b_j = 0.$

\begin{lemm}\label{lem:affine-base} For any total order $<_\pi$ with $n$ as maximum, or equivalently for any $\pi$ in $S_{n}$, the set $B^{\pi}=\{\:b^\pi_0,\dots,b^\pi_n\:\}$ forms an affine basis of $L_G$.
\end{lemm}

\begin{proof} It is easy to see that the matrix of $\{b_{\pi(0)}^\pi,\dots,b_{\pi(n-1)}^\pi\}$ in the base $B$ is upper triangular with diagonals equal to $1$. It follows that the set $\{b_{\pi(0)}^\pi,\dots,b_{\pi(n-1)}^\pi\}$ is a basis of $L$. As $b^\pi_{\pi(n)}=0$, it follows that $B^\pi$ is an affine basis.
\end{proof}

We denote by $\triangle^{\pi}$ the simplex defined by $B^\pi$. In other words, $\triangle^\pi:=\conv(B^\pi)$, the convex-hull of $B^\pi$.
Consider the fundamental parallelotope $F(B)$ defined by the basis $B$ of $L_G$. Note that $F(B)$ is the convex-hull of all the vectors $b^\pi_i$ for $\pi\in S_n$ and $i\in\{0,\dots,n\}$.  We next show that the set of simplices $\{\triangle^\pi\}_{\pi \in S_{n}}$ provides a simplicial decomposition (i.e., a triangulation) 
of $F(B)$. But before we need the following simple lemma: 
\begin{lemm} \label{freud_lem} Let $\square_n=\{\:(x_0\dots,x_{n-1})\:|\:0\leq x_i\leq 1\}$ be the unit hypercube in $\mathbb R^n$.
For a permutation $\pi\in S_n$, let $\bar\triangle_n^{\pi}=\{\:x=(x_0,\dots,x_{n-1}) \in \mathbb{R}^n\:|\:0 \leq x_{\pi(n-1)} \leq x_{\pi(n-2)}\leq \dots \leq x_{\pi(0)} \leq 1\}$. The set of simplices $\{\:\bar\triangle_n^{\pi}\:\}_{\pi \in S_n}$ is a simplicial decomposition of $\square_n$.  
\end{lemm} 
We have
\begin{lemm}\label{simdecomp_lem}Let $G$ be a connected graph and $L\subset A_n$ be the corresponding Laplacian lattice. The set of simplices $\{\triangle^{\pi}\}_{\pi \in S_{n}}$ is a simplicial decomposition of $F(B)$.
\end{lemm}
\begin{proof}
Since $B$ is a basis of the $n$ dimensional lattice $L_G$, which is contained in $H_0$, it is also a basis of $H_0$. By definition, $F(B)$ is the unit cube with respect to the basis $B$. By Lemma \ref{freud_lem}, the family of simplices 
$\{\bar\triangle^\pi\}_{\pi \in S_{n}}$ is a simplicial decomposition of $F(B)$, where $\bar\triangle^\pi=\{\:x=x_0b_0+\dots+x_{n-1}b_{n-1}) \in H_0\:|\:0 \leq x_{\pi(n-1)} \leq x_{\pi(n-2)}\leq \dots \leq x_{\pi(0)} \leq 1\}$ and the vectors are written in the $B$-basis. 
Now recall that the vertices of $\triangle^\pi$ are given by the points $b^\pi_j$. Recall also that
$
\forall\ i \in \:\{\:0,\dots, n\:\},\: \ \  b^{\pi}_i\::=\: \sum_{j \leq_\pi i}b_j\:,$
and that $b^\pi_n=0$. A simple calculation shows that $\triangle^\pi$ coincides with the simplex $\bar\triangle^\pi$ above, and the proof follows.
\end{proof}

\noindent A combination of this lemma with the simple fact that $F(B)+L_G$ is a tiling of $H_0$ gives us:
\begin{coro} \label{simpdecomp_cor} The set of simplices $\{\:\triangle^{\pi}+p\:|\:\pi \in S_{n},\:p\in L_G\:\}$ forms a triangulation of $H_0$.
\end{coro}
In the simplicial decomposition ${\{\triangle^{\pi}+p\:|\:\pi\in S_n \:\&\: p\in L\}}$ of $H_0$, consider the set $\mathcal Sim_O$ consisting of all the simplices  that contain the origin $O$ as a vertex.  We have
\begin{lemm} \label{S0_lem} A simplex is in $\mathcal Sim_O$ if and only if it is spanned by $B^{\pi}$ for some $\pi$ in $S_{n+1}$. (Remark that we do not assume that $\pi(n)=n$.)
 \end{lemm}
\begin{proof}
By Corollary \ref{simpdecomp_cor}, we know that  every simplex in $\mathcal Sim_O$ is of the form:
$\triangle^{\pi_0}+q$ for some $\pi_0$ in $S_{n}$ and $q$ in $L$. Recall that the element $\pi$ of $S_{n}$ is regarded as an element of $S_{n+1}$, with the property that $\pi(n)=n$. Since, the vertex set of $\triangle^{\pi_0}$ is $V(\triangle^{\pi_0})=\{b^{\pi_0}_0,\dots,b^{\pi_0}_{n-1},O\}$, we should have $q=-b^{\pi_0}_i$ for some $0 \leq i \leq n-1$. Let $0\leq j\leq n$ be such that $\pi_0(j)=i$. A straightforward calculation shows that $V(\triangle^{\pi_0})-b^{\pi_0}_i=V(\triangle^{\pi})$, where $\pi=\pi_0\sigma^{j}$ and $\sigma$ is the cyclic permutation $(0,1,2,..,n) \rightarrow (1,\dots,n,0)$.  The lemma follows because every element  $\pi\in S_{n+1}$ can be written uniquely in the form $\sigma^i\pi_0$ for some $\pi_0\in S_n$ ($S_{n+1}=S_nC_{n+1}$).
\end{proof}
Remark also that  $|\mathcal Sim_O|=|S_{n+1}|=(n+1)!$.

Our aim now will be to provide a complete description of the set $\Ext^c(L_G)$ of extremal points of $\Sigma^c(L)$ (and equivalently the set $\Ext(L_G)=\Ext^c(L_G)-(1,\dots,1)$) in terms of this triangulation. Actually we obtain an explicit description of the set $\Crit V_\triangle(O)$.
Before we proceed, let us introduce an extra notation. Let $\pi$ be an element of the permutation group $S_{n+1}$. We do not suppose anymore that $\pi(n)=n$. We define the point $\nu^\pi\in\mathbb Z^{n+1}$ as the tropical sum of the points of $B^\pi$, i.e., $\nu^\pi \::=\:\bigoplus_{i=0}^n\: b^\pi_i.$ (And recall that $b^\pi_i = \sum_{j\leq_\pi i}b_j$.) We have the following theorem.

\begin{theo}\label{thm:main-laplacian}
Let $G$ be a connected graph and $L_G$ be the Laplacian lattice of $G$.
\begin{itemize}
\item[(i)] The set of extremal points of $\Sigma^c(L_G)$ consists of  all the points $\nu^\pi+p$ for $\pi\in S_{n+1}$ and $p\in L_G$, i.e.,
$\emph{Ext}^c(L_G) = \{\:\nu^\pi+p\:|\:\pi\in S_{n+1} \:\textrm{and}\:p\in L_G \:\}.$
As a consequence, we have $\emph{Ext}(L_G) = \{\:\nu^\pi+p+(1,\dots,1)\:|\:\pi\in S_{n+1} \:\textrm{and}\:p\in L_G \:\}.$ 
\item[(ii)] We have $\emph{Crit} V_\triangle(O) = \pi_0(\{\:\nu^\pi\:|\:\pi\in S_{n+1}\:\})$.
\end{itemize}
\end{theo}

 It is quite easy to see that the set $\{\nu^\pi\}$ has the following properties (c.f. Theorem~\ref{theo:laplacian-unif+reflect} below.) 
\begin{itemize}
\item[(P1)-]{\bf Reflection Invariance.} For all $\pi\in S_{n+1}$, $\nu^\pi+\nu^{\bar \pi}=(-\delta_0,-\delta_1,\dots,-\delta_n)$ where $\bar\pi$ is the opposite permutation to $\pi$, and $\delta_i$ denotes the degree of the vertex $v_i$. Since $\Crit V_\triangle(O) = \pi_0(\{\:\nu^\pi\:|\:\pi\in S_{n+1}\:\})$, it follows that $L_G$ is strongly reflection invariant. More precisely we have $\Crit V_\triangle(O)=-\Crit V_\triangle(O)+\pi_0((-\delta_0,\dots,-\delta_n)).$ (Recall that $\pi_0$ is the projection function.)
\item[(P2)-]{\bf Uniformity.} For all $\pi\in S_{n+1}$, $\deg(\nu^\pi)=-m$. In other words, the Laplacian lattice $L_G$ is uniform. 
\end{itemize}

The proof of the results of this section will be given in the next subsection. However, let us quickly show how to calculate $g$ and $K$ in the above corollary. The vertices $\nu^\pi$ all belong to $\Ext^c$ and have degree $-m$. It follows that the vertices of $\Ext(L) = \Ext^c+(1,\dots,1)$ have all degree $-m+n+1$, and so by the definition of genus, we obtain $g_{min}=g_{max}=m-n$. In particular $g$ coincides with the graphical genus of $G$ (which is the number of vertices minus the number of edges plus one). Since the points of $\Ext(L_G)$ are of the form $\nu^\pi+(1\dots,1)$, and as we saw in the proofs of Theorem~\ref{rrineq_theo} and Theorem~\ref{theo:R-R-uniform}, we have $K = -(\nu^\pi+(1,\dots,1))-(\nu^{\bar\pi}+(1,\dots,1))=(\delta_0-2,\delta_1-2,\dots,\delta_n-2)$. 

\subsection*{Proofs of Theorem~\ref{thm:main-laplacian} and Theorem~\ref{theo:laplacian-unif+reflect}}
\label{app:lemmas5-bis}

It is easy to see that the point  $\nu^{\pi}= \bigoplus_{i=0}^n b^\pi_i$ has the following explicit form: 
\begin{align}\label{b_form}
\nu^{\pi}\:&=\:(-\sum_{j<_\pi 0}b_{j0},-\sum_{j <_\pi 1}b_{j1},\dots,-\sum_{j <_\pi n}b_{nj}).
\end{align}
It follows that
 \begin{align*}
 \nu^{\pi}\:&=\:(-\delta_0 + \sum_{j >_{\pi}0}b_{j0},-\delta_1+\sum_{j>_\pi 1}b_{j1},\dots,-\delta_{n}+\sum_{j >_\pi n}b_{nj})\\
      &=\:(-\delta_0,\dots,-\delta_n)-(-\sum_{j<_{\bar\pi} 0}b_{j0},-\sum_{j <_{\bar\pi} 1}b_{j1},\dots,-\sum_{j <_{\bar\pi} n}b_{nj})\\
&=\:\:(-\delta_0,\dots,-\delta_n)-\nu^{\bar\pi}.
 \end{align*}

And we infer that  
\begin{lemm}\label{lem:ex1}
For every $\pi\in S_{n+1}$, we have $\nu^\pi+\nu^{\bar\pi}=(-\delta_0,\dots,-\delta_n)$.
\end{lemm}

Second, we calculate the degree of the point $\nu^\pi$. It is easy to see that
\[\deg(\nu^\pi)= -\sum_{i,j\::\:j<_\pi i}b_{ij} = -m,\]
where $m$ denotes the number of edges of $G$, or equivalently in terms of the matrix $Q$, $m=\frac 12\sum_i\delta_i=trace(Q)/2$. It follows that
\begin{lemm}\label{lem:ex2}
All the points $\nu^\pi$ have the same degree.
\end{lemm}

 We now show that $\nu^\pi\in\Sigma^c(L)$ for every $\pi\in S_{n+1}$.
Assume for the sake of contradiction that there exists a point $p\in L$ such that $p>\nu^\pi$. 
By the definition of the Laplacian lattice $L$, we know that there are integers $\alpha_0,\dots, \alpha_n$ such that $p=\alpha_0b_0+\dots,\alpha_nb_n$, and so we can write 
$$p=(\sum_{j=0}^n(\alpha_0-\alpha_j)b_{j0},\dots,\sum_{j=0}^n(\alpha_n-\alpha_j)b_{jn}).$$
for some $\alpha_i\in \mathbb Z$.
 Among the integer numbers $\alpha_i$, consider the set of indices $S_{p}$ consisting of the indices $i$ for which $\alpha_i$ is minimum. Remark that as $p$ is certainly non zero (since there is a coordinate of $\nu^\pi$ which is zero, we cannot have $0>\nu^\pi$), we cannot have $S_p = \{0,\dots,n\}$. Now in the set $S_{p}$ consider the index $k$ which is the minimum in the total order $<_\pi$. By construction of $k$, we have $\alpha_k-\alpha_j\leq-1$ for all $j <_\pi k$ and $\alpha_k-\alpha_j \leq 0$ for all $j \geq_\pi k$. It follows that $p_k$, the $k-$th coordinate of $p$, is bounded above by
\begin{align*}
p_k\:&=\:\sum_{j=0}^n(\alpha_k-\alpha_j)b_{jk}\:\leq\:\sum_{j<_\pi k}\:-b_{jk}\:=\:\nu^\pi_k.
\end{align*}

And this contradicts our assumption $p>\nu^\pi$.

Next, we need to show that $\nu^{\pi}$ is a local minimum of the degree function. We already now that $\deg(\nu^\pi)=-m$. We will prove that for every point $x\in \Sigma^c(L)$, we have $\deg(x)\geq -m$. By Lemma~\ref{lem:lower-graph}, it will be enough to prove that $h_{\triangle,L}(x)\leq \frac m{n+1}$ for every point $x\in L$. By the definition of the simplicial distance function $h_{\triangle,L}$, this is equivalent to proving that the simplex $x+\frac{m}{n+1}\triangle$ contains a lattice point $p\in L$, i.e.,
\begin{equation}\label{eq:latticepoints}
\forall x\in H_0,\:\: (x+\frac m{n+1}\triangle) \cap L \neq \emptyset.
\end{equation}

Here we use the following trick to reduce the problem to the case when all the entries of $Q$ are non-zero. We add a rational number $\epsilon=\frac st$, $s,t\in \mathbb N$, to each $ b_{ij}$, $i\neq j$, to obtain $b_{ij}^\epsilon$. We also define $\delta_i^\epsilon$ in such a way that $\sum_i b_{ij}^\epsilon =\delta_i^\epsilon$. Remark that $\sum_j \delta_j^\epsilon = \frac{tr(Q^\epsilon)}2$. The new matrix $Q^\epsilon$ is not integral  anymore (but if we want to work with integral  lattices, we can multiply every coordinate by a large integer $t$ to obtain an integral  matrix $tQ^\epsilon$). If we know that our claim is true for all Laplacians with non-zero coordinates, then the function $h$ associated to $tQ^\epsilon$ satisfies the property 
 \begin{equation}\label{eq:hdelta}
 h_{\triangle,L^{t,\epsilon}} \leq \frac{tr(tQ^\epsilon)}{2(n+1)}.
 \end{equation}

 Where $L^{t,\epsilon}$ denotes the lattice generated by the matrix $tQ^\epsilon$. Let $L^\epsilon$ be the (non 
necessarily integral ) lattice generated by the matrix $Q^\epsilon$. It is easy to see that $t.h_{\triangle,L^\epsilon}=h_{\triangle,L^{t,\epsilon}}$. Equation~\ref{eq:hdelta} implies then
  \begin{equation}\label{eq:h2}
 h_{\triangle,L^{\epsilon}} \leq \frac{tr(Q^\epsilon)}{2(n+1)}=\frac m{n+1}+\frac{n\epsilon}{2(n+1)}.
 \end{equation}
Using characterisation of Equation~\ref{eq:latticepoints}, one can see that, varying $\epsilon$, the above property for all sufficiently small rational $\epsilon>0$ will imply that $h_{\triangle,L}\leq \frac m{n+1}$, and that is what we wanted to prove. 
Indeed one can easily show that the distance function $h_{\triangle,L_{\epsilon}}(p)$ is a continuous function in $\epsilon$ and $p$.

So at present, we have shown that we can assume that all the $b_{ij}$'s are strictly positive. This is the assumption we will make for a while. In this case, using the explicit calculation of $\nu^\pi$, we can quite easily show that 

\begin{lemm}\label{vprop_lem} The point $\nu^{\pi}$ has the following properties:\\
1. $\nu^{\pi}_i=b^\pi_{ii}$ for $0 \leq i \leq n$.\\
2. $\nu^{\pi}_j<b^\pi_{ij}$ for $i \neq j$ and $0 \leq i,j \leq n$. 
\end{lemm}

As a corollary we obtain:
\begin{coro} \label{orth_cor} Let $\{e_0,\dots,e_n\}$ be the standard orthonormal basis of $\mathbb{R}^n$, i.e., $e_0=(1,0,$ $\dots,0)$ $,\dots,e_n=(0,\dots,0,1)$. Let $e_j$ be a fixed vector. For every $\delta>0$, $\nu^{\pi}-\delta e_j<b^\pi_j$. \end{coro}

\begin{lemm}\label{wdom_lem}
For every non-zero vector $w$ in $H^{-}_{O}$ and for every $\delta>0$, there exists a point $p$ in $L$ such that $\nu^{\pi}-\delta w<p$.
\end{lemm}
\begin{proof}
Follows easily from the above discussion.
\end{proof}
It follows now easily that
\begin{coro}\label{cor:critical} The point $\nu^{\pi}$ is an extremal point of $\Sigma^{c}(L)$.
\end{coro}
\begin{proof} Follows by combining Lemmas \ref{wdom_lem} and \ref{dom_lem}. \end{proof}

We will now prove the following: {\it every extremal point of $\Sigma^{c}(L)$ can be written as the tropical sum of the vertices of a simplex of the form $\triangle^\pi+q$, for some $\pi \in S_n$ and some $q$ in $L$.} Again we will first assume a stronger condition that $b_{ij}>0$ for all $i,j$ such that  $i \neq j$ and $1 \leq i,j \leq n-1$.
And then we do a limiting argument similar to the one we did above to obtain the general statement.
Let $\pi \in S_{n}$ a fixed permutation. Using the assumption $b_{ij}>0$ for $i \neq j$, it is easy to show that 

\begin{lemm}\label{abs_lem}\label{ord_lem}  For any total ordering $<_\pi$, $\pi\in S_{n+1}$, we have $b^\pi_{ij}\neq 0$ for all $0 \leq i,j \leq n$ and $i\neq \pi(n)$. (Remark that $b^\pi_{\pi(n)}=0$.) Here $b^\pi_{ij}$ is the $j-$th coordinate of the vector $b^\pi_i$. In addition, if $b^\pi_{ij}>0$ (resp. $b^\pi_{ij} < 0$ ), then $j\leq_{\pi}i$ (resp. $i<_\pi j <_\pi \pi(n)$). 
\end{lemm}
 
As we saw in Lemma~\ref{S0_lem}, the set of simplices $\Delta^\pi$, $\pi \in S_{n+1}$, coincides with $\mathcal Sim_O$, the set of all simplices of the triangulation which are adjacent to $O$. The simplices of $\mathcal Sim_O$ naturally define a fan $\mathcal F$, the maximal elements of which are the set of all cones $\mathcal C^\pi$ generated by $\Delta^\pi$ for $\pi \in S_{n+1}$. In other words if $B^\pi$ denoted the affine basis $\{b^\pi_i\}_{i=0}^n$, the cone $\mathcal C^\pi$ is the cone generated by $B^\pi$. In particular every element of $H_0$ is in some $\mathcal C^\pi$ for some $\pi \in S_{n+1}$. We have
\begin{lemm} \label{neighdom_lem} Let $q$ be a point in $L$, and $q\neq b^\pi_i$ for all $\pi\in S_{n+1}$ and $0\leq i\leq n$. Let $\mathcal C^\pi$ be a cone in $\mathcal F$ which contains $q$. There exists a vector $b^\pi_i$ in $B^\pi$ such that $p<b^\pi_i$ for every point $p$ in $H^{-}_O \cap H^{-}_q$. In particular, no point in $H^{-}_O \cap H^{-}_q$ is contained in $\Sigma^c(L)$.
\end{lemm}
 \begin{proof} Since $q$ is a point in $L \cap \mathcal C^\pi$,  there exists non-negative integers $\alpha_i\geq 0 $, $0\leq k\leq n-1$, such that we can write $q=\sum_{k=0}^{n-1}\alpha_kb^\pi_{\pi(k)}$. In addition, since $q\notin B^\pi$, we have $\sum_l \alpha_l \geq 2$. Let $j=\min\{\:k\:|\: \alpha_k \neq 0\:\}$, i.e., the minimum index such that $\alpha_k \neq 0$, and let $i = \pi(j).$ We show that the point $b^\pi_{i}$ satisfies the condition of the lemma. For this, it will be enough to prove that $b^\pi_i > O\oplus q$. Indeed $p \in H^{-}_O \cap H^{-}_q$ implies that $p \leq O \oplus q$, and so if $b^\pi_i > O\oplus q$, then we have $p < b^\pi_i$, which is the required claim. 

We should prove that $b^\pi_{ik} > (O\oplus q)_k$ for all $k$. As $i = \pi(j)\neq \pi(n)$,  by Lemma \ref{abs_lem} we know that $b^\pi_{ik} \neq 0$ for all $k$. There are two cases: if $b^\pi_{ik}>0$, then easily we have $b^\pi_{ik}> 0\geq (O \oplus q)_k$. If $b^\pi_{ik}<0$, then  by Lemma \ref{ord_lem}, we have 
$b^\pi_{lk}<0$ for all $l \geq_\pi i$.
By the choice of  $i$, we have $\alpha_{l}=0$ for all $l<_\pi i$. We infer that $b^\pi_{jk} > \sum_l \alpha_l b^\pi_{lk} = (O \oplus q)_k$, and the lemma follows. 
\end{proof}

 We obtain the following corollary: {\it the simplices of our simplicial decomposition form the dual of the Voronoi diagram}. More precisely 
\begin{coro}\label{cor:vor-ex} Let $q$ be a point in $L$ that is not a vertex of a simplex in $\mathcal Sim_O$, i.e., $q\neq b^\pi_i$ for all $\pi\in S_{n+1}$ and $0\leq i\leq n$. Then $V(O) \cap V(q)=\emptyset$. Hence, for every two points $p$ and $q$ in $L$, we have $V(p) \cap V(q) \neq \emptyset$ if and only if $p$ and $q$ are adjacent in the simplicial decomposition of $H_0$ defined by $\{\triangle^{\pi}+p\:|\:\pi\in S_n \:\&\: p\in L\}$ i.e., $V(p) \cap V(q) \neq \emptyset$ if and only if there exists $\pi\in S_{n+1}$ such that $q$ is a vertex of $\Delta^\pi +p$.
 \end{coro}
\begin{proof} We prove the first statement by contradiction. So for the sake of a contradiction, assume the contrary and let $p \in V(O) \cap V(q)$. By definition, we have $h_{\triangle,L}(p)=d_{\triangle}(p,O)=d_{\triangle}(p,q) \leq d_{\triangle}(p,q')$ for all points $q'\in L$. By Lemma \ref{lowenv_lem} this implies that the point $y=f_O(p)=f_q(p)$ is a point in $\partial \Sigma^c(L)$ (c.f. Section~\ref{sec:voronoi} for the definition of $f_p$). By the definition of $f_{p}$, the point $y$ is in $H^{-}_O \cap H^{-}_q$.  On the other hand, Lemma \ref{neighdom_lem} implies that no point in $H^{-}(O) \cap H^{-}(q)$ can be contained in $\Sigma^{c}(L)$. We obtain a contradiction. To see the second part, by translation invariance we can assume $p=O$. And in this case, the results follows by observing that for $q\in \Delta^\pi$, $\pi_0(\nu^\pi) \in V(O)\cap V(q)$. 
 \end{proof}

We can now present the proof of Theorem~\ref{thm:main-laplacian} in the case where all the $b_{ij}$'s are strictly positive. It will be enough to prove that $\Crit V(O)=\pi_0(\{\nu^{\pi}\:|\: \pi \in S_{n+1}\})$. As $v_\pi\in H_O^-$ and we showed that $v_\pi$ is in $\Ext^c(L)$, we have $\pi_0(\{\nu^{\pi}\:|\: \pi \in S_{n+1}\}) \subseteq \Crit V(O)$. We show now $\Crit V(O) \subseteq \pi_0(\{\nu^{\pi}\:|\: \pi \in S_{n+1}\})$. Let $v \in \Crit V(O)$ and $x$ be the point in $\Ext^c(L)$ with $\pi_0(x)=v$. By Lemma~\ref{ver_cor}, there exist points $p_0,\dots,p_n\in L$ such that $v \in V(p_0)\cap\dots\cap V(p_n)$. By Corollary~\ref{cor:vor-ex}, points $p_0,\dots,p_n$ should be adjacent in the simplicial decomposition of $H_0$ defined by $\{\triangle^{\pi}+p\:|\:\pi\in S_n \:\&\: p\in L\}$. As $v$ is also in $\Vor(O)$, it follows that one of the $p_i$ is $O$, and so there exists $\pi \in S_{n+1}$ such that $v \in \cap_{p\in B^\pi} \Vor(p)$. By the proof of Lemma~\ref{ver_cor}, we also have $x = \bigoplus_i p_i$. But $\bigoplus_{p\in B^\pi} p=\nu^\pi$. It follows that $x=\nu^\pi$. We infer that $v \in \pi
 _0(\{\nu^{\pi}\:|\: \pi \in S_{n+1}\})$ and the theorem follows.  
\vspace{.3cm}

The proof of Theorem~\ref{theo:laplacian-unif+reflect} is a simple consequence of Lemma~\ref{lem:ex1}, and what we just proved, namely, $\Ext^c(L)=\{\nu^\pi+q \:|\:  \pi \in S_{n} \:\&\:  q\in L\:\}$ and $\Crit V(O)=\pi_0(\{\nu^{\pi}\:|\: \pi \in S_{n+1}\})$.  

\vspace{.3cm}

To prove the general case, it will be enough to show that $\Crit V(O)=\pi_0(\{\nu^{\pi}\:|\: \pi \in S_{n+1}\})$ still holds. Indeed the rest of the arguments remain unchanged.

We consider again the $\epsilon$-perturbed Laplacian $Q_\epsilon$ and do a limiting argument similar to the one we did before. Let $L_\epsilon$ to be the lattice generated by $Q_\epsilon$. By $\Vor(L_\epsilon)$ and $\Crit V_{\epsilon}(p)$, we denote the Voronoi diagram of $L_\epsilon$ under the distance function $d_\triangle$ and the Voronoi cell of a point $p\in L_\epsilon$.  We also define $B^\pi_\epsilon$, $\Delta^\pi_\epsilon$, and $\nu^\pi_\epsilon$ similarly.

Theorem~\ref{thm:main-laplacian} in the case where all the coordinates are strictly positive implies that $\Crit V_{\epsilon}(O)=\pi_0(\{v_{\epsilon}^{\pi}\:\:|\:\: \pi \in S_{n+1}\})$. We can naturally define limits of the sets $\Crit V_{\epsilon}(O)$ as $\epsilon$ tends to zero as limits of the points $\pi_0(v_{\epsilon}^{\pi})$. Indeed this limit exists and coincides with the set $\pi_0(\{\nu^\pi\:\:|\:\:\pi\in S_{n+1}\})$, as can be easily verified. We show now 

\begin{lemm}\label{limiting_lem}
We have $\lim_{\epsilon \rightarrow 0} \emph{Crit} V_{\epsilon}(O) = \emph{Crit} V(O)$.
\end{lemm}
\begin{rema}\rm
 Unfortunately this is not true in general for non graphical lattices. However, we always have 
$\Crit V(O) \subseteq \lim_{\epsilon \rightarrow 0} \Crit V_{\epsilon}(O)$.
\end{rema}

\begin{proof}[Proof of Lemma~\ref{limiting_lem}] By Corollary~\ref{cor:critical}, we already know that every point of $\pi_0(\{\nu^\pi\:\:|\:\:\pi \in S_{n+1}\})$ is critical.
So we should only prove that these are the only critical points, namely $\Crit V(O) \subseteq \lim_{\epsilon \rightarrow 0} \Crit V_{\epsilon}(O)=\pi_0(\{\nu^\pi\:\:|\:\:\pi \in S_{n+1}\})$. Let $c$ be a critical point of $L$. By Lemma~\ref{ver_cor}, we know that there exists a set of points $p_0,\dots,p_n$ such for each $i$, the facet  $F_i$ of $\bar\triangle_{h_{\triangle,L}(c)}(c)$ contains $p_i$ and none of the other points $p_j \neq p_i$. We will show the following: for all sufficiently small $\epsilon$, there exists a point $c_\epsilon \in L_\epsilon$ and $h_\epsilon = h_{\triangle,L_\epsilon}(c_\epsilon) \in \mathbb R_{+}$ such that $\bar\triangle_{h_{\epsilon}}(c_\epsilon)$ has the same property for the lattice $L_\epsilon$, namely, for each $i$, the facet $F_{\epsilon,i}$ of $\bar\triangle_{h_{\epsilon}}(c_\epsilon)$ contains a point $p_{\epsilon,i}\in L_\epsilon$ which is not in any other facet $F_{\epsilon,j}$ of $\bar\triangle_{h_{\epsilon}}(c_\epsilon)$, for $j\neq i$. In addition $\bar\triangle_{h_{\epsilon}}(c_\epsilon)\rightarrow \bar\triangle_{h_{\triangle,L}(c)}(c)$,   and so $h_{\epsilon} \rightarrow h_{\triangle,L}(c)$ and $c_{\epsilon} \rightarrow c$ ($h_\epsilon$ and $c_\epsilon$ being the radius and the centre of these balls $\bar\triangle_{h_{\epsilon}}(c_\epsilon)$). As each of the point $c_{\epsilon}$ will be critical for $L_\epsilon$, we conclude that $c\in \lim_{\epsilon \rightarrow 0} \Crit(L_\epsilon)$ which is easily seen to be enough for the proof of the lemma. To show this last statement, we argue as follows: for small enough $\epsilon$, there exist points $q_{\epsilon,0}$ and $p_{\epsilon,1},\dots,p_{\epsilon,n} \in L_\epsilon$ such that $q_{\epsilon,0} \rightarrow p_{0}$ and for all $n \geq i\geq 1$, $p_{\epsilon,i} \rightarrow p_i$ when $\epsilon$ goes to zero. These points naturally define  a ball for the metric $d_{\bar\triangle}$, i.e., a simplex of the form $\bar\triangle_{r_\epsilon}(\bar c_{\epsilon})$.
  This is the bounded simplex defined by the set of hyperplanes $E_{i,\epsilon}$, where $E_{i,\epsilon}$ is the hyperplane parallel to the facet 
$F_i$ of $\bar\triangle_{h_{\triangle,L}(c)}(c)$ which contains $p_{\epsilon,i}$ ($q_{\epsilon,0}$ for $i=0$). We define the ball $\bar\triangle_{h_\epsilon}(c_\epsilon)$ as follows. For each $\epsilon$, if the interior of $\bar\triangle_{r_\epsilon}(\bar c_{\epsilon})$ does not contain any other lattice point (a point of $L_\epsilon$), we let $\bar\triangle_{h_\epsilon}(c_\epsilon) := \bar\triangle_{r_\epsilon}(\bar c_{\epsilon})$. If the interior of $\bar\triangle_{r_\epsilon}(\bar c_{\epsilon})$ contains another point of $L_\epsilon$, let $p_{\epsilon,0}$ be the furthest point from the hyperplane $E_{0,\epsilon}$ and $E'_{0,\epsilon}$ the hyperplane parallel to $E_{0,\epsilon}$ which contains this point. The simplex (ball) $\bar\triangle_{h_\epsilon}(c_\epsilon)$ is the simplex defined by the hyperplanes $E_{0,\epsilon}$ and $E_{1,\epsilon},\dots,E_{n,\epsilon}$. These simplices have the following properties:
\begin{itemize}
\item For all small $\epsilon$, $\bar\triangle_{h_{\epsilon}}(c_{\epsilon})$ does not contain any point of $L_{\epsilon}$ in its interior. In consequence  $h_{\epsilon} = h_{\triangle,L_{\epsilon}}(c_{\epsilon})$.
\item When $\epsilon \rightarrow 0$, the simplices $\bar\triangle_{h_{\epsilon}}(c_{\epsilon})$ converge to $\bar\triangle_{h_{\triangle,L}(c)}(c)$ (in Gromov-Haussdorf distance for example).
\item The point $p_{\epsilon,0}$ is in the interior of the facet $F_{\epsilon,0}$ of the simplex $\triangle_{h_{\epsilon}}(c_{\epsilon})$. In addition for sufficiently small $\epsilon$, each point $p_{\epsilon,i}$ is in the interior of the facet $F_{\epsilon,i}$ of the simplex $\triangle_{h_{\epsilon}}(c_{\epsilon})$. This is true because $\bar\triangle_{h_{\epsilon}}(c_{\epsilon})\rightarrow \bar\triangle_{h_{\triangle,L}(c)}(c)$, $p_{\epsilon,i} \rightarrow p_i$, and each point $p_i$ is in the interior of the facet $F_i$ of $\bar\triangle_{h_{\triangle,L}(c)}(c)$.  
\end{itemize}

 These properties show that the point $c_{\epsilon}$ is critical for $L_{\epsilon}$ and $\lim_{\epsilon \rightarrow 0}c_{\epsilon} = c$, which completes the proof.
\end{proof}

The proofs of Theorem~\ref{thm:main-laplacian} and Theorem~\ref{theo:laplacian-unif+reflect} are now complete.

 We note that this representation of $c$  as a limit of $c_{\epsilon}$ is not in general unique. Indeed it is quite straightforward to check the following two theorems which show together that: {\it Different non-equivalent classes of critical points, up to linear equivalence, can converge in the limit to the same class.} 
\begin{theo}
In the case where all $b_{ij} > 0$, none of the points $\nu^\pi$ for $\pi \in S_n$ is linearly equivalent to another one, i.e., they define different classes in $\mathbb R^{n+1}/L$. In particular, the number of different critical points up to  linear equivalence is exactly $n!$.  
\end{theo}
However for general graphs this number is usually strictly smaller than $n!$. We state here without proof the following result about the number of non-equivalent classes of critical points. Two permutations $\pi$ and $\sigma \in S_n \subset S_{n+1}$ are {\it elementary
equivalent} if $\pi$ is obtained from $\sigma$ by switching two consecutive vertices in the order defined by $\sigma$ which are not adjacent in $G$ ($\pi(n)=\sigma(n)=n$). Two elements $\pi$ and $\sigma \in S_n$ are {\it equivalent} if there is a sequence of elementary equivalences which relate $\pi$ to $\sigma$.
Each equivalent class for this  equivalence relation is called a {\it cyclic order} of $G$.
\begin{theo}
 Let $G$ be a given connected graph on $n+1$ vertices. The number of different critical points up to linear equivalence for the Laplacian lattice $L$ is exactly the number of different cyclic orders of $G$.
\end{theo}
    
\begin{rema}\rm
The number in the above theorem is exactly the number of acyclic orientations of $G$ with a unique fixed source $v_0$, which is also the evaluation of the Tutte polynomial at the point $(1,0)$~\cite{Stan}.
\end{rema}

\subsection{Lattices Generated by Laplacian of Connected Regular Digraphs}\label{sec:digraphs}
 In this section, we briefly describe how to extend partially the results of the previous section to connected regular digraphs. A digraph $D$ is {\it regular} if the in-degree and out-degree of each vertex are the same.  This allows to define a Laplacian matrix for $D$, almost similar as in the graphic case: if the vertices of $D$ are enumerated by $\{v_0,\dots,v_n\}$, the matrix representation of the Laplacian $D$ is of the form Equation~\ref{gra_form} but we do not have symmetry any more. Namely 
\begin{equation}\label{gra_form_di}
Q=
\begin{bmatrix}
  \delta_0      & -b_{01} & -b_{02}  \hdots & -b_{0n} \\
 -b_{10}   &  \delta_1    & -b_{12}  \hdots & -b_{1n} \\
 \vdots    &  \vdots & \ddots\\
 -b_{n0}   &  b_{n1} & -b_{n2}  \hdots &  \delta_n
\end{bmatrix}
\end{equation}
has the following properties:
\begin{itemize}
\item[$(C_1)$] $b_{ij}$'s are integers and $b_{ij}\geq0$ for all $0\leq i \neq j\leq n$.
\item[$(C_2)$] $\delta_i=\sum_{j=1,j\neq i}^{n}b_{ij}=\sum_{j=1,j\neq i}^{n}b_{ji}$ (and is the in-degree ($=$ out-degree) of the vertex $v_i$).
\end{itemize}
 We obtain the simplicial decomposition of $H_0$ defined by $\{\triangle^\pi + p\:\:|\:\: \pi \in S_n \:\: \textrm{and}\:\: p \in L\}$, similar to the case of unoriented graphs.
In the case where all the coordinates $b_{ij}$ are strictly positive, we can similarly prove the following results (the proofs remain unchanged):
\begin{itemize}
\item For all $\pi \in S_n$ and $p\in L$, the point $\nu^\pi+p$ is extremal (c.f. Corollary~\ref{cor:critical}).
\item For every two points $p$ and $q$ in $L$, we have $V(p) \cap V(q) \neq \emptyset$ if and only if  $p$ and $q$ are adjacent in the simplicial decomposition of $H_0$ defined by $\{\triangle^{\pi}+p\:|\:\pi\in S_n \:\&\: p\in L\}$. In other words, $V(p) \cap V(q) \neq \emptyset$ if and only if there exists $\pi\in S_{n+1}$ such that $q$ is a vertex of $\Delta^\pi +p$ (c.f. Corollary~\ref{cor:vor-ex}).
\item The set of extremal points of $\Sigma^c(L_G)$ consists of  all the points $\nu^\pi+p$ for $\pi\in S_{n+1}$ and $p\in L_G$, i.e.,
$\Ext^c(L_G) = \{\:\nu^\pi+p\:|\:\pi\in S_{n+1} \:\textrm{and}\:p\in L_G \:\}.$
As a consequence, we have $\Ext(L_G) = \{\:\nu^\pi+p+(1,\dots,1)\:|\:\pi\in S_{n+1} \:\textrm{and}\:p\in L_G \:\}.$ More precisely, we have $\Crit V_\triangle(O) = \pi_0(\{\:\nu^\pi\:|\:\pi\in S_{n+1}\:\})$. (c.f. Theorem~\ref{thm:main-laplacian}).
\item We have $g_{min}	 = -\max_{\pi\in S_n} \deg(\nu^\pi) - n$ and $g_{max} = -\min_{\pi\in S_n} \deg(\nu^\pi) - n$.
\item Riemann-Roch Inequality. Remark~\ref{rema:R-R-digraph} can be applied: for $K = (\delta_0 - 2,\dots, \delta_n-2)$, we have for all $D$, $$g_{min}-\deg(D)-1\:\leq\:r(K-D)-r(D)\leq g_{max}-\deg(D)-1.$$
\end{itemize}

 In the general case, where some of the $b_{ij}$'s could be zero, unfortunately the limiting argument does not behave quite well. Indeed, there are examples of regular digraphs for which a point $\nu^\pi$ is not a critical point for $L$ for some $\pi \in S_n$.  However as the proof of Lemma~\ref{limiting_lem} shows, we always have 
$\Crit(L) \subseteq \lim_{\epsilon \rightarrow 0} \Crit (L_\epsilon).$ So it could happen that we lose  (strong) reflection invariance. Although we do not know in general if such lattices have any sort of reflection invariance, it is still possible to prove a Riemann-Roch inequality for these lattices by taking the limit of the Riemann-Roch inequalities for the lattices $L_\epsilon$. One point in doing this limiting argument is to extend the definition of the rank function to all the points of $\mathbb R^{n+1}$ (and not only for integral  points), which we briefly defined in the beginning of this paper. This new rank-function will have image in $\{-1\} \cup \mathbb R_{+}$ and is continuous on the points where it is strictly positive.

 In the general case we have the following results:
\begin{itemize}
\item Every point of degree $-\min_{\pi\in S_n} \deg(\nu^\pi)$ among the points $\nu^\pi$ is extremal (by a similar limiting argument as in the graphic case). So we have $g_{max} = -\min_{\pi\in S_n} \deg(\nu^\pi) - n$. In addition, $g_{min} \geq -\max_{\pi\in S_n} \deg(\nu^\pi) - n$. Let $\bar{g}_{min} = -\max_{\pi\in S_n} \deg(\nu^\pi) - n$.
\item (Riemann-Roch Inequality.) Taking the limit of the family of inequalities $g^{\epsilon}_{min}-\deg(D)-1 \leq r_{\epsilon}(K_{\epsilon}-D)-r_{\epsilon}(D)\leq g^{\epsilon}_{max}-\deg(D)-1,$ where $\epsilon$ goes to zero, we get
$$\bar g_{min}-\deg(D)-3\:\leq\:r(K-D)-r(D)\leq g_{max}-\deg(D)+1.$$
\end{itemize} 
Here $r_\epsilon$ is the rank function for the lattice $L_\epsilon$. This is because $\lim_{\epsilon \rightarrow 0} g^{\epsilon}_{min}=\bar g_{min}$; $\lim_{\epsilon \rightarrow 0} g^{\epsilon}_{max}=g_{max}$; and $r(E) + 1 \geq \lim_{\epsilon \rightarrow 0} r_{\epsilon}(E) \geq r(E) - 1$ for all $E \in \mathbb R^{n+1}$.

\subsection{Two Dimensional Sub-lattices of $A_2$}
In this section, we consider full-rank sub-lattices of $A_2$. 
First, we show that all these sub-lattices are reflection invariant. This is indeed an easy consequence of Theorem~\ref{theo:crit-n!} by which a sub-lattice of $A_2$ of rank two has at most two different classes of critical points. It follows that:
\begin{theo}
Every sub-lattice $L$ of $A_2$ of dimension two is reflection invariant.
\end{theo}

 Indeed something quite strong holds in dimension two: every two dimensional sub-lattice $L$ of $A_2$ is a Laplacian lattice of some regular digraph on three vertices.

\begin{lemm}\label{digraph_lem}
Every full dimensional sub-lattice of $A_2$ is the Laplacian lattice of a regular digraph on three vertices. 
\end{lemm}
 
 Let $\{e_0,e_1,e_2\}$ be the standard basis of $H_0$ where $e_0=(2,-1,-1)$,  $e_1=(-1,2,-1)$ and $e_2=(-1,-1,2)$. Let the linear functional $g_0,g_1$ and $g_2$ be defined by taking the scaler product with $e_0,e_1,e_2$ respectively. So for example for $u = (u_0,u_1,u_2)$, $g_0(u)=2u_0-u_1-u_2$. Let $b_0,b_1$ be a basis of $L$ and $b_2 = -b_0-b_1$. Let $Q$ be the matrix having $b_0,b_1$ and $b_2$ as its first, second and third row, respectively. For $i=0,1,2$, define the cone $C_i$ to be the set of vectors $v$ such that $g_i(v) \geq 0$ and $g_j(v)\leq 0$ for $j\neq i$. We have 
\begin{lemm}
The basis $b_0,b_1,b_2$ is the basis defined by a regular digraph if and only if the following holds: for each $i$, $b_i$ is in the cone $C_i$.
\end{lemm}
\begin{proof}
Let $b_{ij}$ denote the $j$-th coordinate of $b_i$. It will be enough to show that $b_{ij} \leq 0$ for $i\neq j$. Let $j'\in \{0,1,2\}$ be different from $i$ and $j$. We have $g_j(b_i) \leq 0$. But $g_{j}(b_i) = 2b_{ij}-b_{ii}-b_{ij'} = 3b_{ij}$. It follows that $b_{ij}\leq 0$. 
\end{proof}

\begin{proof}{\bf of Lemma~\ref{digraph_lem}}
We should show the existence of lattice points $\{b_0,b_1,b_2\}$ such that:
\begin{itemize}
\item[(i)] $\{b_0,b_1\}$ is a basis of $L$;
\item[(ii)] $b_0+b_1+b_2=O$;
\item[(iii)] $b_i$ is contained in the cone $C_i$.
\end{itemize}
 First consider a shortest vector $b_0$ of the lattice and a shortest vector of the lattice $b_1$ that is linearly independent of $b_0$. Using for example Pick's formula, one can show that $\{b_0,b_1\}$ forms a basis of the lattice $L$. We may now assume that $b_0$ is contained in one of the cones $C_0,C_1$ or $C_2$, and without loss of generality $C_0$. Indeed if $b_0$ does not belong to any of these cones then $-b_0$ will belong to one of these cones, and we may replace $b_0$ by $-b_0$. So we assume that $b_0$ belongs to $C_0$. It is well known that $b_1$ can be chosen such that the angle between $b_0$ and $b_1$ is in the interval $[\frac \pi3,\frac{2\pi}3]$. Since the maximum angle between any two points in $C_i$ is $\frac \pi3$, $b_1$ is contained in a cone different from $C_0$ and $-C_0$. Now, if $b_1$ is not contained in $C_1$ or $C_2$ then $-b_2$ will be  in $C_1$ or $C_2$, and we can replace $b_1$ by $-b_1$. Remark that $\{b_0,-b_1\}$ will remain a basis. Hence, we may assume without loss of generality that $B=\{b_0,b_1\}$ is a basis of the lattice such that $b_0$ is contained in cone $C_0$ and $b_1$ is contained in cone $C_1$. \\
This means that $b_0=(b_{00},b_{01},b_{02})$ and $b_1=(b_{10},b_{11},b_{12})$, where $b_{01},b_{02},b_{10},b_{12}\leq 0$ and $b_{00} = -b_{01}-b_{02} >0$ and $b_{11} = -b_{10}-b_{12}>0$. First, we observe that $-b_3 = b_0+b_1$ is contained in $C_0 \cup C_1 \cup -C_2$, and if it is in $-C_2$, then we have our set of lattice points $\{b_0,b_1,b_2\}$. We now define a procedure which, by updating the set of vectors $b_0,b_1$, provides at the end the set of lattice points $\{b_0,b_1,b_2\}$ with properties $(i), (ii)$ and $(iii)$ above. 
The procedure is defined as follows:
\begin{itemize}
\item[(a)] If $b_0+b_1 \in -C_2$ then stop.
\item[(b)] Otherwise, if $b_0+b_1 \in C_0$ replace $b_0$ by $b_0+b_1$ and iterate.
\item[(c)] Otherwise, if $b_0+b_1 \in C_1$ replace $b_1$ by $b_0+b_1$ and iterate.
\item[(d)] Output $\{b_0,b_1, b_2\}$, where $b_2=-b_0-b_1$.
\end{itemize}
 We will show that the number of iterations is finite. And this shows that the final output has the desired properties. Indeed, at each iteration $\{b_0,b_1\}$ form a basis of $L$ (if $\{b_0,b_1\}$ is a basis of $L$ then $\{b_0+b_1,b_1\}$ and $\{b_0,b_0+b_1\}$ will also be a basis of $L$), and so by the definition of the procedure, the finiteness of the number of steps shows that at the end we should have $b_0+b_1 \in -C_3$.
To show that the procedure terminates after a finite number of iterations, consider a step of the algorithm: if the step $(b)$ in the procedure happens, then $b_0+b_1$ should be in $ C_0$ and not in $-C_2$. This means that $0>b_{01}+b_{11}$, which implies that $|g_1(b_0+b_1)|<|g_1(b_0)|$. Indeed $g_{1}(b_0+b_1) = 3b_{01}+3b_{11} <0$ and so $|g_{1}(b_0+b_1)| = -3b_{01}-3b_{11} < -3b_{01} = |g_{1}(b_0)|$. Furthermore, we have, $0\leq g_0(b_0+b_1) \leq g_0(b_0)$, since $g_0(b_1) \leq 0$.

Similarly, if the step $(c)$ in the above procedure happens, then $b_0+b_1$ should be in $C_1$ and not in $-C_2$. Hence, we should have $|g_0(b_0+b_1)|<|g_0(b_1)|$ and $0\leq g_1(b_0+b_1) \leq g_1(b_1)$. 
We infer that, starting form $b_0$ and $b_1$, at each iteration one of the two inequalities $|g_1(p)|<|g_1(b_0)|$ or $|g_0(p)|<|g_0(b_1)|$ for $p=b_0+b_1$ should be satisfied. Furthermore, at every iteration we have $|g_0(p)|\leq |g_0(b_0)|$ and $|g_0(p)|\leq |g_0(b_1)|$. 
Hence, an upper bound on the number of iterations is the number of lattice points  $p$ in $C_0$ with $|g_0(p)| \leq |g_0(b_0)|$ plus the number of lattice points $q$ in $C_1$ with $|g_1(q)|\leq |g_1(b_1)|$ and this is indeed finite.
\end{proof}

\begin{rema}\rm
In higher dimensions, the analogue of Lemma~\ref{digraph_lem} is unlikely to be true since a simple calculation shows that the minimum angle between cones $C_i$ and $C_j$ is at least $\pi/3$ (here, as in dimension two $e_0,\dots,e_n$ is the corresponding basis of $H_0$ where $e_0=(n,-1,\dots,-1),\dots, e_n=(-1,\dots,-1,n)$, and $g_i$ is the linear form defined by taking the scaler product with $e_i$). Indeed, let $p=(\sum_{i\neq 0} p_i,-p_1,\dots,-p_n) \in C_0-\{O\}$ and $q=(-q_0,\sum_{i\neq 1}q_i,-q_2,\dots,-q_n) \in C_1-\{O\}$. We have
\begin{align*}
\frac {p \cdot q}{|p|_{\ell_2}|q|_{\ell_2}}&=\frac{-\sum_{i\neq 0}p_iq_0-\sum_{i\neq 1}q_ip_1+p_2q_2+\dots+p_nq_n}{|p|_{\ell_2}|q|_{\ell_2}}\\
& \leq \frac{p_2q_2+\dots+p_nq_n}{|p|_{\ell_2}|q|_{\ell_2}}\\ 
&\leq \frac {p_2q_2+\dots+p_nq_n}{2\sqrt{p^2_2+\dots+p^2_n}\sqrt{q^2_2+\dots+q^{2}_n}} \leq \frac 12.
\end{align*}

The two inequalities of the last line follow from the set of inequalities 
$$|p|_{\ell_2}=\sqrt{ (p_1+\dots+p_n)^{2}+p_1^2+\dots+p_n^{2}} \geq \sqrt{2(p_1^{2}+\dots+p_n^{2})} \geq \sqrt{2(p_2^{2}+\dots+p_n^{2})}$$
$$|q|_{\ell_2} \geq \sqrt{q_2^2+\dots+q_n^{2}} \:\:\:\: \textrm{(similarly as above)},$$ and the Cauchy-Schwartz inequality. Hence, if the lattice $L$ is generated by a regular digraph, then there exists a basis such that the pairwise angles between the elements of the basis is at least $\frac \pi3$. But, it is known that there exist lattices that are not weakly orthogonal, see \cite{NellDaBar07}. However, note that the notion of a weakly-orthogonal lattice seems to be slightly different from the notion of a digraphical lattice.
\end{rema}

We now characterise all the sub-lattices of $A_2$ which are strongly reflection invariant. 
\begin{theo}\label{thm:A2-strongreflect} A sub-lattice $L$ of $A_2$ is strongly reflection invariant if and only if there are two different classes of critical points up to linear equivalence or $L$ is defined by a multi-tree on three vertices (i.e., a graph obtained from a tree by replacing each edge by multiple parallel edges).
\end{theo}
\begin{proof}
Let $\{b_0,b_1\}$ be the regular digraph basis of $L$ and $b_2=-b_0-b_1$. We consider the triangulation $\{\triangle^\pi+p\}$ of $H_0$ defined by this basis. Let $T$ be the triangle defined by the convex hull of $\{0,b_0,b_0+b_1\}$ ($=\triangle^\pi$) and let $\bar T$ be the opposite of $T$, the triangle defined by the convex hull of $\{0,b_2, b_1+b_2\}$ ($=\triangle^{\bar \pi}$), and let $c_T$ and $c_{\bar T}$ be $\pi_0(\nu^\pi)$ and $\pi_{0}(\nu^{\bar \pi})$. At least one of the points $c_T$ or $c_{\bar T}$ is critical. And in addition the set of critical points of $\Crit V(O)$ is a subset of $\{c_{T},c_{T}-b_0, c_{T}+b_2, c_{\bar T},c_{\bar T}+b_0, c_{\bar T}-b_2\}$.

($\Rightarrow$)
If $c_T$ and $c_{\bar T}$ are both critical points and they are different, we have $\Crit V(O)=\{c_{T},c_{T}-b_0, c_{T}+b_2, c_{\bar T},c_{\bar T}+b_0, c_{\bar T}-b_2\}$ and we can directly see that $-\Crit V(O)=\Crit V(O)+t$ where $t=c_{T}+c_{\bar T}$. It is easy to check directly than the only case when $c_T$ and $c_{\bar T}$ are equivalent is when $b_0 = (a,0,-a)$ and $b_{1} = (0,b,-b)$ for $a,b>0$ (in which case $c_{T} = \pi_0((0,0,-a-b)$ and $c_{\bar T} = \pi_0((-a,-b,0))$, and so $c_{\bar T}-c_T = b_2$ and the lattice is also uniform). In this case, we also have $\Crit V(O)=\{c_{T},c_{T}-b_0, c_{T}+b_2, c_{\bar T},c_{\bar T}+b_0, c_{\bar T}-b_2\}$ and so again $-\Crit V(O)=\Crit V(O)+t$ where $t=c_{T}+c_{\bar T}$.

($\Leftarrow$) If there is just one critical point up to linear equivalence, let us assume without loss of generality that the critical point is $c_T$. In this case, $\Crit V(O)=\{c_T,c_T-b_0,c+b_2 \}$.  It is now easy to check that for any bijection $\phi$ of $\Crit V(O)$ onto itself, $x+ \phi(x)$ cannot be the same over all $x$ in $\Crit V(O)$. 
\end{proof}

We end this section by providing an example of a sub-lattice $L$ of $A_2$ which is not strongly reflection invariant. By the previous theorem, $L$ should contain only one critical point up to linear equivalence and should not be a multi-tree. (In particular, since we only have one class of critical points, $L$ is uniform and satisfies the Riemann-Roch theorem.)

Consider the two dimensional sub-lattice of $A_2$ defined by the vectors $b_0 = (7,-7,0)$ and $b_1=(-3,11,-8)$, and let $b_2=-b_0-b_1=(-4,-4,8)$. These vectors form the rows of the $3\times 3$ matrix $Q$ (which is the Laplacian matrix of a regular digraph). 
\begin{equation}
Q=
\begin{bmatrix}
  \hspace{0.3cm}7   && -7  && \hspace{0.3cm}0 \\
 -3   && \hspace{0.3cm}11  && -8 \\
 -4   &&    -4  && \hspace{0.3cm}8 \\
 \end{bmatrix}
\end{equation} 
Let $\pi$ and $\bar \pi$ be the permutation corresponding to the order $0<_\pi 1<_\pi 2$ and its opposite $1<_{\bar \pi} 0<_{\bar\pi}2$ as in the proof of Theorem~\ref{thm:A2-strongreflect}. We have $\nu^\pi=\bigoplus\{b_0,b_0+b+1,O\}=(0,-7,-8)$ and $\nu^{\bar\pi}=\bigoplus\{b_1,b_1+b_0,O\}=(-3,0,-8)$. We claim that $\nu^{\bar\pi}$ is not an extremal point of $\Sigma^c(L)$ and so $\pi_0(\nu^{\bar \pi})$ is not critical. This is true because  $\nu^{\bar \pi}+(7,-7,0) = (4,0,-8) \geq \nu^{\pi}$, and so $\nu^{\bar \pi}$ cannot be extremal.

Indeed the above example can be turned into a generic class of examples, that we now explain. Consider a lattice defined by generators of the form $b_0=(\alpha,-\alpha,0)$ and $b_1 = (-\gamma, \gamma+\eta, -\eta)$:
\begin{equation}
Q=
\begin{bmatrix}
  \hspace{0.3cm}\alpha   && -\alpha  && \hspace{0.3cm}0\\
  -\gamma  && \hspace{0.3cm}\gamma+\eta  && -\eta \\
  \gamma-\alpha   && -\gamma+\alpha-\eta && \hspace{0.3cm}\eta\\
 \end{bmatrix}
\end{equation} 
Here we suppose in addition that $\alpha,\gamma,\eta > 0$ and $\gamma < \alpha\leq \eta+\gamma$ such that the above matrix is the Laplacian of a regular digraph.
The two permutations $\pi$ and $\bar\pi$ are defined as above, so for these permutations we have 
$\nu^\pi = (0,-\alpha,-\eta)$ and $\nu^{\bar \pi} = (-\gamma,0,-\eta)$. It is clear that $\deg(\nu^\pi) < \deg (\nu^{\bar\pi})$. We infer that $\nu^\pi$ is extremal. But $\nu^{\bar \pi}$ is not extremal since  $\nu^{\bar \pi}\geq \nu^\pi-b_0$. It is also easy to see that $L$ cannot have a multi-tree basis.

\subsection{Examples of sub-lattices with Riemann-Roch property which are not graphical}
In this subsection, we show that there exist an infinite family of sub-lattices $\{L_n\}_{n=2}^{\infty}$, where $L_n$ is a full rank sub-lattice of $A_n$, each $L_n$ satisfies the Riemann-Roch theorem (we say that it has the Riemann-Roch property), and such that none of $L_n$ is graphical. By not being graphical, we mean that there does not exist any basis of $L$ which comes from a connected unoriented multi-graph, i.e. $L_n \neq L_G$ for any connected multi-graph $G$ on $n+1$ vertices. Indeed, we have already provided in the previous section such an example (and even an infinite number of them) in dimension two: the family of sub-lattices of $A_2$ defined by $b_0=(\alpha,-\alpha,0)$ and $b_1 = (-\gamma, \gamma+\eta, -\eta)$ (we will prove this shortly below). The construction of $L_n$ for larger values  of $n$ is then recursive. Suppose we have already constructed an infinite family of full rank sub-lattices of $A_{n}$ which are not graphical and have the Riemann-Roch property, and let 
 $L_{n}$ be an element of this family. Then we construct a full rank sub-lattice of $A_{n+1}$ as follows. By taking the natural embedding $A_{n}\subset A_{n+1}$, $(x_0,\dots,x_{n}) \rightarrow (x_0,\dots,x_{n},0)$, we embed $L_{n}$ in $A_{n+1}$. The lattice $L_{n+1}$ is obtained by adding $b_{n} = (0,0,\dots,0,-1,1)$ to the image of $L_n$. Remark that if $L_{n}$ comes from a regular digraph $G$ with vertices $v_0,\dots,v_{n}$, then $L_{n+1}$ is the lattice of the digraph $G'$ consisting of $G$ and a new vertex $v_{n+1}$ which is connected to $v_{n}$ by two arcs, one in each direction.  We will see that $L_{n+1}$ will not be graphical, and in addition it will have the Riemann-Roch property. Here we provide the details of the construction.

\subsection*{The Lattices $L_2$}
Let $L_2$ be a sub-lattice of $A_2$ defined by $b_0=(\alpha,-\alpha,0)$ and $b_1 = (-\gamma, \gamma+\eta, -\eta)$, where $\alpha,\gamma,\eta > 0$ and $\gamma < \alpha\leq \eta+\gamma$.
\begin{prop} The sub-lattice $L_2$ has Riemann-Roch property and $L_2$ is not graphical. 
\end{prop}

\begin{proof}
We saw in the previous section that $L_2$ has only one class of critical points, up to linear equivalence, is not strongly reflection invariant, and in addition $|\Crit V(O)|=3$. This shows that $L_2$ cannot be graphical. However, $L_2$ is uniform and reflection invariant, and so it has the Riemann-Roch property.  
\end{proof}

\subsection*{The Lattices $L_{n}$}

Let $L_{n}$ be a full rank sub-lattice of $A_{n}$ that we regard as an $n$-dimensional sub-lattice of $A_{n+1}$ by taking the embedding $A_{n} \subset A_{n+1}$ described above. Define $L_{n+1}$ to be the lattice generated by $L_n$ and $b_{n+1} = (0,\dots,0,-1,1)$. We first provide two correspondences: one between the rank function $r_n$ of $L_n$ and the rank function $r_{n+1}$ of $L_{n+1}$, and the other one, between the extremal points of $L_n$ and the extremal points of $L_{n+1}$. 

\noindent Let $D$ be an element of $\mathbb{Z}^{n+2}$. By $D|_n$ we denote the projection of $D$ to $\mathbb Z^{n+1}$ obtained by eliminating the last coordinate. So if $D=(D_0,\dots,D_{n+1})$, then $D|_n = (D_0,\dots,D_n)$.

\begin{lemm} Let $D=(D_0,\dots,D_{n+1})$ be a point in $\mathbb{Z}^{n+2}$ and let $D'=(D-D_{n+1}b_{n+1})|_{n+1}$. We have $r_{n+1}(D)=r_{n}(D')$. 
\end{lemm} 
\begin{proof} We first prove that $r_{n}(D') \geq r_{n+1}(D)$. Let $E'\in \mathbb Z^{n+1}$ be effective.
We should prove that if $\deg(E')\leq r_{n+1}(D)$, then $D'-E' \geq q'$ for at least one $q' \in L_n$. Let $E = (E',0)$. As $\deg(E) \leq r_{n+1}(D)$, there exists a $q \in L_{n+1}$ such that $D-E \geq q$. By the definition of $L_{n+1}$, there exists $q'\in L$ and $\alpha \in \mathbb Z$ such that $q= (q',0) + \alpha b_{n+1}$. It follows that 
$D_{n+1} \geq \alpha$, and so $D'-E' = (D-D_{n+1}b_{n+1}-E)|_n \geq (D-\alpha b_{n+1}-E)|_{n} \geq q'$. So $D'-E' \geq q'$ and we are done. 

We now show that $r_{n+1}(D)\geq r_{n}(D')$. Let $E = (E_0,\dots,E_{n+1}) \in \mathbb Z^{n+2}$ be effective of degree at most $r_n(D')$. We have to prove the existence of a point $q \in L_{n+1}$ such that $D-E \geq q$. Let $O \leq E' \in \mathbb Z^{n+1}$ be defined by $E-E_{n+1}b_{n+1} = (E',0)$. In other words $E' =(E_0,\dots,E_{n-1}, E_n+E_{n+1})$. It is clear that $E' \geq O$ and $\deg(E') \leq r_n(D')$. So there exists a point $q'\in L_n$ such that $D' - E' \geq q'$. We infer that $D - E \geq (q',0)+(D_{n+1}+E_{n+1})b_{n+1}$. So for $q = (q',0)+(D_{n+1}+E_{n+1})b_{n+1} \in L_{n+1}$, we have $D-E \geq q$, and we are done.
\end{proof}

\begin{lemm}\label{extconst_lem} The extremal points of $\Sigma(L_{n+1})$ are of the form $(v,0)+q$ where $v$ is an extremal point $\Sigma(L_n)$ and $q$ is a point in $L_{n+1}$. Similarly, the elements of ${\emph Ext}^{c}(L_{n+1})$ are of the form $(u,-1)+q$ where $u$ is a point of ${\emph Ext}^{c}(L_{n})$ and $q \in L_{n+1}$.
\end{lemm}
\begin{proof}
The proof is similar to the proof of the previous lemma and we only prove one direction, namely $\Ext(L_{n+1}) \subseteq \Ext(L_n)\times\{0\} + L_{n+1}$. The other inclusions
$\Ext(L_n)\times\{0\} + L_{n+1} \subseteq \Ext(L_{n+1})$, $\Ext^c(L_{n+1}) \subseteq \Ext^c(L_n)\times\{-1\} + L_{n+1}$, and $\Ext^c(L_n)\times\{-1\} + L_{n+1} \subseteq \Ext^c(L_{n+1})$ follows similarly.

Let $\bar v = (\bar v_0,\dots,\bar v_{n+1})$ be an extremal point of $L_{n+1}$, i.e., $\bar v \in \Ext(L_{n+1})$. Let $v \in \mathbb Z^{n+1}$ be defined as follows: $(v,0) = \bar v-\bar v_{n+1}b_{n+1}$. The claim follows once we have shown that $v$ is an extremal point of $L_n$. To prove that $v$ is an extremal point, we need to show that for all $q \in L_{n}$, $v \nleq q$ and that  $v$ is a local minimum for the degree function. Suppose that this is not the case and let $q \in L_{n}$ be such that $v \leq q$. We have $\bar v \leq (q,0)+\bar v_{n+1}b_{n+1}$ and $(q,0)+\bar v_{n+1}b_{n+1} \in L_{n+1}$, which is a contradiction to the assumption that $\bar v \in \Ext(L_{n+1})$. The proof that $v$ is a local minimum follows similarly.
 \end{proof}
As a corollary to the above lemmas, we obtain

\begin{coro} If $L_{n}$ has the  Riemann-Roch property (resp. is uniform and reflection-invariant), then $L_{n+1}$ also has the Riemann-Roch property (resp. is uniform and reflection-invariant). Furthermore, we have $K_{n+1}=(K_n,0)$, where $K_i$ is canonical for $L_i$, $i=n,n+1$. 
\end{coro}
We now show that if $L_2$ is the family of lattices that we described above, then $L_n$ is not graphical. Indeed, by applying Lemma~\ref{extconst_lem} and by induction on $n$, it is easy to show that $L_{n}$ is not strongly reflection invariant, provided that $L_2$ is not strongly reflection invariant, and we know that this is the case.

Remark that the family of all $L_n$ constructed above is infinite (for each fixed $n$). Indeed, by using the fact that $\Pic(L_n) = \Pic(L_{n+1})$, and by observing that the set $|\Pic(L_2)|$ contains an infinite number of values, we conclude that $|\Pic(L_n)|$ takes an infinite number of values and so the family of all $L_n$ is infinite.	

\section{Algorithmic Issues}\label{sec:algo}
Let $L$ be a full-rank sub-lattice of $A_n$. While it is not well known whether calculating the rank of a given point for a Laplacian lattice can be done in polynomial time or not, calculating the rank function for general $L$ becomes more complicated. In this section, we show that this problem is {\sc NP}-hard. Actually we prove that deciding if $r(D) \geq 0$ is already {\sc NP}-hard for general $D$ and $L$. Remark that for the case of graphs, deciding if $r(D) \geq 0$ can be done in polynomial time~\cite{H99, Shokrieh09}. 

By the results of Section~\ref{sec:prel}, deciding if $r(D)=-1$ is equivalent to deciding whether $-D \in \Sigma(D)$ or not. So we will instead consider this membership problem. We will show below that this problem is equivalent to the problem of deciding whether a rational simplex contains an integral  point. We then use this to show that it is generally {\sc NP}-hard to decide if a given integral  point $D$ is contained in $\Sigma(L)$. As every point of positive degree is in $\Sigma(L)$, we may only consider the points of negative degree.

We first state the following simple lemma.
\begin{lemm}\label{lat_lem} Let $D$ be a point in $\mathbb{Z}^{n+1}$ of negative degree. We have $D \in \Sigma(L)$ if and only if the simplex $\bar{\triangle}_{\frac{-\deg(D)}{n+1}}(\pi_0(D))$ contains no lattice point (a point in  $L$) in its interior.
\end{lemm}
\begin{proof} We saw in Section~\ref{sec:voronoi} that $\partial \Sigma^c(L)$ is the lower graph of the function $h_{\triangle,L}$. It follows that $D \in \Sigma(L)$ if and only if $-\frac{\deg(D)}{n+1} < h_{\triangle,L}(\pi_0(D))$. By the definition of $h_{\triangle,L}$, this means that $D \in \Sigma(L)$ is equivalent to $d_{\bar\triangle}(p,\pi_0(D))=d_{\triangle}(\pi_0(D),p)>-\frac{\deg(D)}{n+1}$ for all $p \in L$, which is to say, $\bar{\triangle}_{\frac{-\deg(D)}{n+1}}(\pi_0(D))$ contains no lattice point.
\end{proof}

Hence, the question of deciding whether if $D \in \Sigma(L)$ is equivalent to the following question:
\begin{center}
{\it Given a simplex of the form $\bar{\triangle}_r(x)$ with centre at $x$ and radius $r \geq 0$, can we decide if there is a lattice point in the simplex?}
\end{center}
A simple calculation shows that the vertices of $\bar{\triangle}_{-\frac {\deg(D)}{n+1}}(\pi_0(D))$ are all integral. This shows that with respect to the lattice $L$, the simplex $\bar{\triangle}_{-\frac {\deg(D)}{n+1}}(\pi_0(D))$ is rational, i.e., there exists a large integer $N$ such that $N\bar{\triangle}_{-\frac {\deg(D)}{n+1}}(\pi_0(D))$
is a polytope with vertices all in $L$. (This is because $L$ is full dimensional and itself integral.)

 We now recall that the complexity of deciding if an arbitrary rational $n$-dimensional simplex in $\mathbb R^n$ contains a point of $\mathbb Z^n$ is {\sc NP}-hard when the dimension $n$ is not fixed, and  it is polynomial time solvable when the dimension is fixed~\cite{Barvinok05}. In our case, we are fixing the rational simplex, and $L$ is an arbitrary sub-lattice of $A_n$. However, it is quite easy to reduce the original problem to our case and to obtain the same complexity results in our setting. A polynomial-time reduction is described below:

\noindent Given the vertices $V(S)=\{v_1,\dots,v_n\}$ of a rational simplex $S$ in $\mathbb{R}^{n}$, we do the following.

\begin{itemize}
\item[1.] Compute the centroid $c(S) = \frac {\sum_i v_i}{n+1}$ of $S$ and let $S' = S-c(S)$.
\item[2.] Define the linear map $f$ from $\mathbb R^n$ to $H_0$ by sending $V(S')$ bijectively to $V(\bar \triangle) = \{e_0,\dots,e_n\}$. Let $\bar \triangle(x)$ be the image of $S$, where $x = f(c(S))$.
\item[3.] Let $L_0 = f(\mathbb Z^n)$ and $N$ be a large integer such that $NL \subset A_n$ (such $N$ exists since $f$ and $S$ are rational, and so $L$ is rational). Remark that we have $NL \cap N\bar \triangle(x) \neq\emptyset$ if and only if $S \cap \mathbb Z^n \neq \emptyset$. Remark also that $ N\bar\triangle(x)=\bar\triangle_{N}(Nx)$.
\item[4.] Let $D$ be the integral point in $\mathbb Z^{n+1}$ defined by $D=Nx-N(n+1)(1,\dots,1)$. Then $\pi_0(D)=Nx$, $\deg(D) = -N(n+1)$, and $N\bar\triangle(x) = \bar\triangle_{-\frac{\deg(D)}{n+1}}(\pi_0(D))$.)
\end{itemize}
For $L$ defined as above, we infer that $\bar\triangle_{-\frac{\deg(D)}{n+1}}(\pi_0(D))\cap L \neq \emptyset$ if and only $S\cap \mathbb Z^{d}\neq \emptyset$.  
So we have

\begin{theo} For an arbitrary full rank sub-lattice $L$ of $A_n$, the problem of deciding if $r(D)=-1$ given a point $D \in \mathbb{Z}^{n+1}$ and a basis of $L$ is {\sc NP}-hard.
\end{theo}

As a consequence, we also note that the decision version of the problem of computing the rank is {\sc NP}-hard.

\begin{theo} Given an integer $k \geq -1$, a point $D \in \mathbb{Z}^{n+1}$ and a basis of $L$ of a sub-lattice of $A_n$. The problem of deciding if $r(D) \geq k$ is NP-hard.
\end{theo}

It is interesting to note that for the case of Laplacian lattices of graphs on $n+1$ vertices, with a given basis formed by the $n$ rows of the Laplacian matrix, the problem of deciding if an integral point belongs to the Sigma-Region can be done in polynomial time~\cite{H99}.  So we are naturally led to the following questions:
\begin{question} Given a full rank sub-lattice $L$ of $A_n$, does there exist a special basis $B$  of $L$ such that if $L$ is given with $B$, then the membership problem for the Sigma-Region of $L$ can be solved in polynomial time $?$
\end{question} 
\begin{question}
Given a Laplacian sub-lattice of $A_n$, is it possible to find the special basis of $L$ in time polynomial in $n\:?$ Given a sub-lattice of $A_n$, is it possible to decide if $L$ is Laplacian in time polynomial in $n\:?	$  
\end{question}

\section{Concluding Remarks}
In this section, we provide some concluding remarks on the results of the previous sections. 
\subsection{Extension to Non-Integral Sub-Lattices}
It is possible to extend the results of the previous sections to an arbitrary sub-lattice $L$ of $H_0$ of dimension $n$. However, since $\Sigma(L)$ does not make sense in the general case, one needs to work directly with $\Sigma^{\mathbb R}(L)$ and its closure $\Sigma^c(L)$. Thus, in this setting and for integral sub-lattices, min- and max-genus, rank-function, canonical point, etc change but the new theory is easily related to what we considered in the previous sections (through the relation between $\Sigma^{c}(L)$ and $\Sigma(L)$). 

In the general setting, the definition of the rank-function is inspired by the characterization given in Lemma~\ref{ranksig_lem}. Namely, for a point $D$ in $\mathbb{R}^{n+1}$, $\bar r(D)=-1$ if and only if $-D$ is a point in $\Sigma^{c}(L)$. More generally, $\bar r(D)+1$ is the distance of $-D$ to $\Sigma^{c}(L)$ in the $\ell_1$ norm, i.e.,
\begin{align*}
 \bar r(D) \:&\: = dist_{\ell_1}(-D,\Sigma^{c}(L))-1:=\inf\{||p+D||_{\ell_1}\ | \ p\in\Sigma^{c}(L)\}-1. 
\end{align*}

In the case of integral sub-lattices of $A_n$, the previous rank-function $r(.)$ is related to the new rank function by $r(D) = \bar r(D+(1,\dots,1))$. The structural theorem of Sigma-Region remains valid, and the definitions of the max- and min-genus and the distance function extend without change to this case. A Riemann-Roch theorem can be proved for uniform and reflection invariant real sub-lattices of $H_0$.  Note that for integral sub-lattices of $A_n$, the $*-$genus with respect to $\Sigma^{c}$ is $\bar g_{*} = g_{*}+n+1$ where $*\in \{min, max\}$, and the new canonical point $\bar K$ associated to $\bar r$, if exists, is nothing but $K+(2,\dots,2)$, where $K$ is the canonical point described in the previous sections.

\subsection{On the Number of Different Classes of Critical Points.}
Given a full rank sub-lattice $L$ of $A_n$, we bound here the number of different critical points modulo $L$.

\begin{theo}\label{theo:crit-n!}
 For a given sub-lattice $L$ of $A_n $, there are at most $n!$ different critical points modulo $L$.
\end{theo}

\begin{proof}(Sketch of the proof.)
 Let us consider the gradient flow of the function $h$ and its corresponding flow complex. i.e., for each critical point $v \in \Crit(L)$, we have a maximal open subset $U_v \subset H_0$ such that for each point $u\in U_v$, the gradient flow of $h$ starting at $u$ ends at $v$. Moreover, $H_0\setminus \cup_v U_v$ has measure zero. By translation invariance, the tiling obtained by $U_v$ is also translation invariant, i.e., $U_{v+p}=U_v+p$ for all points $p \in L$ and $v\in \Crit(L)$.  Lemma~\ref{ver_cor} implies that for each critical point $v$, there exist points $p_0,\dots,p_{n} \in L$, such that $p_i$'s are in different facets of $\bar\triangle_{h_{\triangle,L}(v)}(v)$. By the definition of the distance function $h_{\triangle,L}$, it is easy to see the simplex $\bar\triangle_{h_{\triangle,L}(v)}(v)$ containing these points is in the topological closure of the open set $U_v$. It follows that $U_v$ has volume at least the volume of the simplex obtained by taking the convex hull of the points $p_0,\dots,p_n$. This volume is at least the volume of the minimal simplex defined by $L$, i.e.,
  $\frac {\textrm{vol}(L)}{n!}$. We infer that each open set $U_v$ has volume at least $\frac {\textrm{vol}(L)}{n!}$. By taking the quotient modulo $L$, we conclude that $|\Crit(L)/L| \leq n!$, i.e., the number of different classes of critical points modulo $L$ is at most $n!$. 
\end{proof}

\subsection{A Duality Theorem for Arrangements of Simplices} \label{sec:arrangement}
Let $L$ be a sub-lattice of $A_n$ of dimension $n$. For a real number $t \geq 0$, define the  arrangement $\A_t$ as the union of all the simplices  $\triangle_t(c)$ for $c \in \Crit(L)$, i.e., 
$$\A_t:=\bigcup_{c \in \Crit(L_G)} \triangle_t(c).$$
A second arrangement $\B_t$ is defined as the union of all the simplices $\bar \triangle_t (p)$ for $p\in L$, i.e., 

$$B_t:=\bigcup_{p \in L_G} \bar{\triangle}_t(p).$$
 (Recall that $\bar \triangle = -\triangle$.)

\begin{defi}
\rm The covering number of a lattice $L$ denoted by $\Cov(L)$ is the smallest real $k$ such that $B_k = H_0$. 
\end{defi}
It is not difficult to show that for a sub-lattice $L$ of $A_n$, the Covering Number is given by  $\Cov(L)=\frac{g_{max}+n}{n+1}$. (Thus, for a uniform lattice, $\Cov(L) = \frac {g+n}{n+1}$.)

Let $G$ be an undirected graph on $n+1$ and with $m$ edges (thus, $g=m-n$).  Let $L_G$ be the Laplacian lattice of $G$. (In this case, by the results of Section~\ref{sec:laplacian} (c.f., Equation~\ref{eq:latticepoints}), the covering number of $\Cov(L_G)$ is the density of the graph.)

\vspace{.3cm}

 The two arrangements $\A$ and $\B$ are dual in the following sense.

\begin{theo}[Duality between $\A$ and $\B$]\label{prop:dual}  For any $0 \leq t \leq \Cov(L_G)$, the arrangement $\B_t$ is the closure of the complement of the arrangement $\A_{\Cov(L_G)-t}$ in $H_0$, i.e., 
 $$\B_t=\Bigl(H_0\setminus \A_{\Cov(L_G)-t}\Bigr)^c.$$
 In particular, for any $0 \leq t \leq \Cov(L_G)$, $\partial{\B_t}=\partial{\A_{\Cov(L_G)-t}}$. 
\end{theo}

\begin{proof}(Sketch of the proof) 
Let $x \in \B_t \cap \A_{\Cov(L_G)-t}$. By the definition of the two arrangements $\B$ and $\A$, there exists a point $p\in L_G$ and a point $c \in \Crit(L_G)$ such that $x \in \bar{\triangle}_t(p)\cap \triangle_{\Cov(L_G)-t}(c)$. By the triangle inequality for $d_{\triangle}$ and the results of Section~\ref{sec:laplacian}, it follows that $d_{\triangle}(c,p) = \Cov(L_G)$, $d_\triangle(x,p) = t$, and $d_{\triangle}(c,x) = \Cov(L_G)-t$. Thus, we have $x \in \partial \B_t \cap \partial \A_{\Cov(L_G)-t}$. It follows that $\B_t$ and $\A_{\Cov(L_G)-t}$ have disjoint interiors, and so $\B_t \subseteq\Bigl(H_0\setminus \A_{\Cov(L_G)-t}\Bigr)^c.$ The other inclusion $H_0\setminus \A_{\Cov(L_G)-t} \subset \B_t$ follows from the structural theorem of the Sigma-Region, Theorem~\ref{dom_thm} (and Theorem~\ref{cdom_thm}). Namely, we claim that for every point $x\in H_0$, there exists a point $p\in L_G$ and a point $c \in \Crit(L_G)$, such that $d_\triangle(c,x) + d_\triangle(x,p)=d_\triangle(c,p) = \Cov(L_G)$, and this clearly implies the inclusion $H_0\setminus A_{\Cov(L_G)-t} \subset B_t$. Let $p$ be a point of $L_G$ such that $h_\triangle(x) = d_\triangle(x,p)$. By Proposition~\ref{lem:lower-graph}, the point $x - h_\triangle(x)(1,\dots,1)$ lies on the boundary of $\Sigma^c$. By Theorem~\ref{cdom_thm}, there exists an extremal point $\nu$ of $\Ext^c(L_G)$ such that $\nu \leq x - h_\triangle(x)(1,\dots,1)$. Let $c$ be the critical point $\pi_0(\nu) \in \Crit(L_G)$.  Note that $h_\triangle(c) = \Cov(L_G)$. By Proposition~\ref{lem:lower-graph}, we have $\nu = c -\Cov(L_G) (1,\dots,1)$. Thus, we have $c - h_\triangle(c) (1,\dots,1) \leq x - h_\triangle(x)(1,\dots,1)$, or equivalently $c - (\Cov(L_G)-h_\triangle(x)) (1,\dots,1) \leq x$. By the explicit definition of $d_\triangle$, we have $d_\triangle(c,x) \leq \Cov(L_G)-h_\triangle(x) = \Cov(L_G) - d_\triangle(x,p)$. Since $d(c,p) \geq \Cov(L_G)$, this shows that $d_\triangle(c,x) = \Cov(L_G) - d_\triangle(x,p)$ and the claim follows. 
\end{proof}

\paragraph*{Acknowledgments.} The presentation has been benefited very much from very helpful and interesting suggestions and remarks of an anonymous referee. The authors are extremely grateful to him/her. Special thanks to Ivan Popov for his help in programming issues which allowed to test some of the conjectures. Many thanks to Kurt Melhorn for his interest in this work, and to Jan van den Heuvel for interesting discussions.

\end{document}